\documentclass[10pt,a4paper]{amsart}

\usepackage{latexsym,amssymb,amsmath,amsthm,amsfonts,enumerate,verbatim,xspace,
exscale}
\usepackage{graphicx}
\usepackage{color,amsbsy,textcomp}
\usepackage{enumerate}
\usepackage{float} 

\usepackage{subcaption}

\usepackage{orcidlink} 

\usepackage{hyperref}
\hypersetup{
    colorlinks=true,
    linkcolor=blue,
    filecolor=magenta,      
    urlcolor=cyan,
}

\input xy 
\xyoption{all} 
\CompileMatrices
\UseComputerModernTips

\theoremstyle{plain}
\newtheorem{theorem}{Theorem}[section]

\newtheorem{proposition}[theorem]{Proposition}
\newtheorem{corollary}[theorem]{Corollary}

\theoremstyle{definition}

\newtheorem{remark}[theorem]{Remark}

\theoremstyle{definition}

\usepackage[
backend=biber,style=alphabetic, natbib=false, mcite=false, casechanger=auto,
	sorting=nyt, sortcites=false, pluralothers=true, maxnames=4, minnames=1,
	backref=false, arxiv=abs, 
	date=year,
	isbn=false, url=false, doi=true, eprint=true, related=false,
	giveninits=true
]{biblatex}
\renewbibmacro{in:}{}	

\def\C{\mathcal{C}}
\def\R{\mathbb{R}}

\def\N{\mathbb{N}}

\newcommand{\up}{\upshape}

\def\vv<#1>{\langle#1\rangle}

\newcommand{\pr}{\mbox{$\text{\up{pr}}$}}

\providecommand{\det}{\mbox{$\text{\up{det}}\,$}}

\providecommand{\vol}{\mbox{$\text{\up{vol}}$}}

\providecommand{\del}{\partial}

\newcommand{\eps}{\varepsilon}

\newcommand{\U}{\mbox{$\mathcal{U}$}}

\newcommand{\VaR}{\mbox{$\textup{VaR}$}}
\newcommand{\EV}{\mbox{$\textup{E}$}}
\newcommand{\var}{\mbox{$\textup{Var}$}}
\newcommand{\cov}{\mbox{$\textup{Cov}$}}

\usepackage[normalem]{ulem}
\usepackage{xcolor}

\newcommand\deleteF{\bgroup\markoverwith{\textcolor{blue}{\rule[0.8ex]{2pt}{0.9pt}}}\ULon}
\newcommand\deleteS{\bgroup\markoverwith{\textcolor{green}{\rule[0.8ex]{2pt}{0.9pt}}}\ULon}

\usepackage[foot]{amsaddr}

\title[On the asymptotic shape of quantile surfaces]%
{On the asymptotic shape of quantile surfaces}

\author[Florian Gach, Simon Hochgerner]{Florian Gach\,\orcidlink{0000-0003-2835-8226}${}^1$, Simon Hochgerner\,\orcidlink{0000-0002-3978-3706}${}^{*,1}$ 
}

\address{ ${}^1$ Austrian Financial Market Authority (FMA),  Otto-Wagner Platz 5, A-1090 Vienna;} 
\address{${}^*$ Corresponding author;}

\email{Florian.Gach@fma.gv.at}
\email{Simon.Hochgerner@fma.gv.at}

\thanks{\emph{Disclaimer.} 
The opinions expressed in this article are those of the authors and do not necessarily reflect the official position of the Austrian Financial Market Authority. } 
\date{September 29, 2026}

\begin{document}

\begin{abstract}
This article is concerned with the asymptotic shape of quantile surfaces, defined as the set of quantiles at a given level $\alpha$ generated by a controlled one-dimensional distribution.
Specifically, when the distribution arises as a linear combination of log-normal random variables and the control is a vector of positive coefficients, we prove that quantile surfaces are globally concave in the left tail ($\alpha\to0$) and globally convex in the right tail ($\alpha\to1$). Moreover, these surfaces exhibit asymptotic separation of scale and shape.
\end{abstract}

\maketitle

\tableofcontents

\section*{Introduction}
\subsection*{Motivation}
In the banking, insurance and general finance industry Value-at-Risk ($\VaR$) is the predominantly used risk measure, in many jurisdictions its use is also a regulatory requirement. The popularity of $\VaR$ is due to its simplicity -- its interpretation can be understood immediately at an intuitive level. 

On the other hand, $\VaR$ is often criticized for the following reasons:
\begin{itemize}
\item[(A)]
it does not always reward risk diversification (lack of sub-additivity);
\item[(B)]
portfolio optimization with respect to $\VaR$ may suffer from multiple local extrema (lack of convexity); 
\item[(C)] 
value at risk at confidence level $\alpha$ does not provide any information on the potential loss beyond $\alpha$; 
\end{itemize}
Item (B) is a consequence of (A). 

Criticism of $\VaR$ has been made systematic from an axiomatic perspective in the highly influential article \cite{Artzner_etal_99}, where also the concept of coherent risk measures was introduced.  
These objections were recognized to be substantial in the report \cite{BCBS2011RiskMeasurement}, and the Basel Committee on Banking Supervision has since moved away from Value-at-Risk and instead adopted Expected Shortfall as the regulatory market risk measure in Basel~III (scheduled to be enforced in the European Union per January 2027). Expected Shortfall  satisfies sub-additivity and is a coherent risk measure. 
Efficient allocation of risk is an important topic in financial applications, and in view of point (B) optimization of Expected Shortfall was studied in \cite{RockafellarUryasev_1998} with the motivation of guaranteeing convexity properties.
On the other hand, the European insurance industry does use Value-at-Risk as its standard risk measure, in particular for market risk, via Solvency~II (\cite{L1}, adopted in 2016). 

The motivation for the present article was to understand how non-convexity of $\VaR$ arises, and ideally to identify regions where convexity can be guaranteed with a view towards characterizing (locally or globally) $\VaR$-optimal asset allocation. Concretely, we are interested in a linear portfolio with pay-off $h(u, X) = \sum_{i=1}^n u_i X_i$ where $X_i$ are risk factors and $u_i>0$ are investment choices. For elliptical  $X_i$ it is known that the quantile at level $\alpha$ of $h(u,X)$, $q_{\alpha}(u) = \VaR_{\alpha}[h(u,X)]$, is concave in $u$ for small $\alpha$ and convex in $u$ for sufficiently large $\alpha$. On the other hand, for log-normally distributed $X_i$ it is unknown how convexity might arise or break down. Given the practical relevance of log-normal risk factor modelling in financial applications, it is our goal to understand convexity properties of $q_{\alpha}(u)$ for log-normal $X_i$.

\subsection*{Description of results}
It turns out that this problem is intimately tied to the geometry of hypersurfaces moving through the tails of the pay-off distribution.  In this regard, our contributions are as follows: 
\begin{enumerate}
\item 
For nonlinear $h$ we provide a formula for the gradient and the Hessian of $q_{\alpha}$ with respect to $u$ (Theorem~\ref{thm:q-der}).   
\item 
For linear $h$ and log-normal $X_1,\dots,X_n$ we show that  $q_{\alpha}(u)$ is concave in $u$ for small $\alpha$ and convex in $u$ for large $\alpha$. 
More precisely, we show that for all $u$ with $u_i>0$ there exists a level $\alpha_0$ such that $q_{\alpha}$ is strictly concave for all $\alpha\le\alpha_0$, when there is not too much positive correlation between log-factors, $\log(X_i)$. On the other hand, when correlation between log-factors is high, we simplify the presentation by restricting to $n=2$ and show that concavity still holds in the small $\alpha$-regime (Theorem~\ref{thm:conc}). Similarly, for high $\alpha$ we also restrict to $n=2$ and show that convexity holds (Theorem~\ref{thm:conv}).     
\item 
In the deep tail regime, for linear pay-off and log-normal risk factors, we find that $q_{\alpha}(u)$ develops a universal structure, where scale and shape are separated and the location parameter of the log-factors is forgotten: as $\alpha\to0$ and as $\alpha\to1$ we explicitly identify (respectively) level functions, $c = c(\alpha)$, and shape  functions, $s=s(u)$, such that $q_{\alpha}(u)\sim c(\alpha)s(u)$. See Section~\ref{sec:univ}.
\end{enumerate}

When the risk factors possess a density which is defined on $\R^n$ then the Hessian formula in Theorem~\ref{thm:q-der} can be decomposed into expectations of: the Hessian of $h$, transport of quadratic deviation from conditional mean, transport of conditional density,  and a mean curvature term (cf.\ equation~\eqref{e:Hess_q_2}), where Hessian and curvature are with respect to the control variable $u$. This is of independent interest, since it shows the possible effects that can occur for nonlinear $h$. For linear $h$ only the transport of the conditional surface density survives, and the rest of the paper is devoted to this term. 

First, in Section~\ref{sec:ell_X}, we consider the case of elliptical risk factors and rederive known convexity properties. 

Section~\ref{sec:log_X} then focuses on $h(u,X) = \sum_{i=1}^n u_i X_i$ with log-normal $X_i$ and $u_i>0$. The approach to analyzing the Hessian formula consists of three steps: we realize the conditional variance derived in Theorem~\ref{thm:q-der} as a surface integral over a moving hyperplane; then we express this integral with respect to a moving frame of reference so that it becomes an integral over a fixed hyperplane; for the $n=2$ cases, we use asymptotic analysis to evaluate the leading terms of this integral as $\alpha\to0$ and as $\alpha\to1$. The main technical point in this process is to describe the (moving frame pull-back of the) conditional measure, $\mu_y$, as a function of $y=q_{\alpha}(u)$ at fixed $u$. It turns out that the measure flow, $y\mapsto\mu_y$, has surprisingly rich structure, and it is this redistribution of mass that is responsible for the transition from concavity in the left tail to convexity in the right tail.  See Theorems~\ref{thm:conc} and \ref{thm:conv}. 

Regarding criticisms (A) and (B) we thus arrive at the refinement that quantiles of linear portfolios with log-normal risk factors are concave/convex in the tails, whence in particular sub-additivity follows. 

Convexity is relevant with regard to optimization (in the sense of minimizing $\VaR$ in the loss tail).
Moreover, the universality structure in Section~\ref{sec:univ} shows that quantile surfaces associated to linear portfolios with log-normal risk factors in the tail tend to organize in a scale invariant way. Thus optimizing $\VaR_{\alpha_0}$ for an $\alpha_0$ (far enough) in the tail is also (nearly) optimal for all other $\alpha$ around and beyond $\alpha_0$. Since optimizers are -by definition- stable with respect to first order perturbations, this yields, for all practical purposes, optimization in the full tail. Criticism (C) remains valid in the sense that loss potential beyond $\alpha_0$ cannot be quantified by $\VaR_{\alpha_0}$, but (nearly) optimal investment can be specified simultaneously for the full tail. 

Section~\ref{sec:OptVaR} contains some comments on the optimization of Value-at-Risk. However, while $\VaR$-optimization is the motivation for the current paper it is not the subject itself. A more thorough study of this topic is currently in preparation.  

Finally, we note that in the economically generic case (not very strong positive correlation between $\log(X_i)$), the shape function $s(u)$ is strictly concave as $\alpha\to0$, whence the penalty from not being at the optimal allocation can be  estimated from a non-zero second derivative, which is also explicitly given via the asymptotic shape structure.   

Regarding potential generalizations to other distributions, it is noteworthy that our asymptotic analysis makes heavy use of the specific form of the log-normal density. On the other hand, the pull-back approach to evaluating the Hessian in Theorem~\ref{thm:q-der} may also work for more general (quasi-convex) distributions.

\subsection*{Related literature}
The formulas in Theorem~\ref{thm:q-der} for the gradient and Hessian of $q_{\alpha}(u)$ for linear $h$ were originally derived in \cite{Gourieroux_etal_2000}, where these results were used for sensitivity analysis.

It was shown in \cite{DanielssonJorgensenMandiraSamorodnitskyDeVries2005} that Value-at-Risk is sub-additive in the tails for distributions with slowly varying tails. Therein it is also argued that the tail is the relevant region for risk management purposes. Notice that log-normal tails are sub-exponential but not slowly varying. 

In \cite{Gulisashvili_Tankov_2016} the authors study the tail behaviour of linear combinations of log-normal random variables and provide an explicit characterization of the tail asymptotics of the density and the distribution functions. Particularly, for the left tail they find that the asymptotics depends on the correlation structure of the Gaussian log-factors, this observation is also true for our concavity result (Theorem~\ref{thm:conc}).
Moreover, they give an asymptotic formula for $\alpha\mapsto q_{\alpha}(u)$ (formulated as a Value-at-Risk expansion \cite[Theorem~6]{Gulisashvili_Tankov_2016}) as $\alpha\to0$ and at fixed $u$, but they do not study convexity properties of $q_{\alpha}(u)$ as a function of $u$. In this sense \cite{Gulisashvili_Tankov_2016} are concerned with the asymptotic location of the quantile while we are concerned with its asymptotic shape.

The asymptotic tail behaviour of sums of log-normal risk factors was further studied by \cite{Embrechts_Hashorva_Mikosch_2014}, where the authors also compared asymptotic Value-at-Risk and asymptotic Expected Shortfall. 

Accurate and efficient determination of a portfolio's Value-at-Risk can be quite challenging in practice (\cite{CrepeyFrikhaLouziSpence2026}). In this context we note that the separation of scale and shape (which we observe for linear log-normal portfolios) may be particularly valuable in applications. This leads to a two-step procedure, where the asymptotic shape determines the optimal allocation which is relevant for risk management purposes, and then the $\VaR$-location, which is relevant for regulatory capital requirements,  can be calculated at the (nearly) optimal allocation with the desired precision.    

Positive aspects of $\VaR$ include its suitability for robust forecasting and backtesting (\cite{Gneiting_2011, Ziegel_2016}). 

Finally, questions of quantile convexity are relevant beyond financial applications, e.g.\  
\cite{Miller_et_al_2025_chance_peak} study optimization of bounds of $\VaR$ in the context of aviation.

\section{The controlled quantile function}\label{sec:QF} 
Consider a smooth function $h: \mathcal{U}\times\mathcal{C} \to\R$, $(u,x)\mapsto h(u,x)$ where $\mathcal{U}\times\mathcal{C}\subset\R^d\times\R^n$ is an open and connected subset. Consider further a random variable $X$ with values in $\mathcal{C}$, with smooth density $f_X: \mathcal{C}\to\R$ and distribution function $F_X$. 

For a given level $\alpha\in(0,1)$, we define the controlled quantile function associated to $Y = h(u, X)$, with  density $f_Y = f_Y(y,u)$ and distribution function $F_Y = F_Y(y,u)$,  as
\begin{equation}
 q_{\alpha}(u)
 = 
 \textup{inf}\,\{
 y\in\R: F_Y(y, u)\ge \alpha 
 \}
\end{equation}
We are interested in the convexity properties of this function with respect to $u$ at fixed $\alpha$.  

For a function $\xi: \mathcal{C}\to\R$ we will abbreviate the conditional expectation as $\EV_{(u,y)}[\xi] := \EV[\xi(X)| h(u,X)=y]$. For  vector valued functions $\xi, \zeta: \mathcal{C}\to\R^k$ we denote conditional covariance by 
$\cov_{(u,y)}[\xi, \zeta] = \EV_{(u,y)}[\xi\otimes\zeta] - \EV_{(u,y)}[\xi]\otimes \EV_{(u,y)}[\zeta]$ where $\xi\otimes\zeta = \xi \cdot \zeta^{\top}$ is the tensor product of vectors, and the variance is accordingly $\var_{(u,y)}[\xi] = \cov_{(u,y)}[\xi,\xi]$. Moreover, it will be convenient to abbreviate the tensor product as $\xi^2 = \xi\otimes\xi$.  

In the following we use $\nabla$ for differentiation with respect to $x$, while differentiation with respect to $u$ will be given by variational notation, e.g.\ $\delta h/\delta u$.

Let $\pr_1: \U\times\C\to\U$ be the projection onto the first factor and consider the mapping $\pi = (\pr_1, h): \U\times\C \to \U\times\R$, $(u,x)\mapsto(u,h(u,x))$. When $(u, y)$ is a regular value of $\pi$ then the preimage $S_{(u,y)} = \pi^{-1}(u,y)$ is a submanifold in $\U\times\C$. In this case, the conditional expectation may be expressed through the co-area formula, that is
\begin{equation}   
\label{e:cond}
 \EV_{(u,y)}[\xi] 
 = f_Y(y,u)^{-1}
 \int_{S_{(u,y)}}\xi(x)\frac{f_X(x)}{|\nabla h(u,x)|}\,d\sigma_{(u,y)}(x)
\end{equation}
for a smooth function $\xi: \C\to\R$ and where $d\sigma_{(u,y)}$ is the induced volume on $S_{(u,y)}$ from the Euclidean volume. 
With $\xi = 1$ we get a formula for $f_Y$, 
\begin{equation}
    f_Y(y,u) = \int_{S_{(u,y)}} \frac{f_X(x)}{|\nabla h(u,x)|}\,d\sigma_{(u,y)}(x)
\end{equation}
Since the set of regular values is open, this formula allows to parametrize $f_Y$ in a neighborhood of $(u,y)$. When we assume additionally that $\EV[D^{(k)}h(u,X)]<\infty$ for all $u\in\U$, where $D^{(k)}$ is any constant partial differential operator on $\U\times\C$ and $k\in\N\cup\{0\}$, then it follows that $f_Y$ is smooth in $(y,u)$.

For linear $h$ the following result is due to \cite{Gourieroux_etal_2000}.

\begin{theorem}[Quantile derivatives]\label{thm:q-der}
Assume that $\EV[(D^{(k)}h(u,X))^m]<\infty$ for all $u\in\U$, all $k\in\N\cup\{0\}$ and $m=1,2$. Let $\alpha\in(0,1)$ and assume that $(u,q_{\alpha}(u))$ is a regular value of $\pi$. 
Then the gradient is given by
\begin{equation}
\label{e:grad}
 \frac{\delta q_{\alpha}}{\delta u}
 = \EV_{q_{\alpha}(u)}\Big[\frac{\delta h}{\delta u}\Big]
\end{equation}
and the Hessian is given by
\begin{equation}
\label{e:Hess_q}
 \frac{\delta^2 q_{\alpha}}{\delta u^2}
 = 
 \EV_{q_{\alpha}(u)}\Big[\frac{\delta^2 h}{\delta u^2} \Big]
 - 
 f_Y(q_{\alpha}(u),u)^{-1}\,
 \frac{\del}{\del y}\Big|_{y = q_{\alpha}(u)}\,
 \Big(
 f_Y(y,u)\var_y\Big[
  \frac{\delta h}{\delta u}
 \Big] 
 \Big)
\end{equation}
\end{theorem}

\begin{proof}
Ad \eqref{e:grad}. 
Since $f_Y$ is smooth, so is the distribution function $F_Y$. Therefore, smoothness of $q_{\alpha}$ follows from the implicit function theorem, and moreover differentiating $F_Y(q_{\alpha}(u),u) = \alpha$ with respect to $u$ yields
\begin{equation}
\label{e:impl_der}
 f_Y\frac{\delta q_{\alpha}}{\delta u} + \frac{\delta F_Y}{\delta u} 
 = 0
\end{equation}
In order to compute $\delta F_Y / \delta u$ we observe first that
\begin{equation} \label{e:Heaviside}
    F_Y(y, u) = \EV[ H(y - h(u, X)) ],
\end{equation}
where $H$ is the Heaviside function given by $H(t) = 1$ for $t > 0$ and $H(t) = 0$ otherwise. 
To differentiate this expression, we fix a mollifier $\rho\ge 0$. That is,  $\rho\in C^{\infty}(\R)$ with compact support $[0, 1]$, and $\int_{\mathbb{R}} \rho = 1$. 
Let $\rho_{\varepsilon}(t) = \frac{1}{\varepsilon} \rho(\frac{t}{\varepsilon})$ and set
\begin{equation}
    H_{\varepsilon}(t)
    := \int_{-\infty}^t \rho_{\varepsilon}(s) \,ds
\end{equation}
Define the mollified distribution function by
\begin{equation}
    F_Y^{\varepsilon}(y, u) 
    = \EV\Big[ H_{\varepsilon}(y - h(u,X)) \Big]
    = \Big(\rho_{\eps}\ast F_Y(.,u) \Big) (y) .
\end{equation}
Note that $F_Y^{\varepsilon}\to F_Y$ in $C^k_{\textup{loc}}(\R\times\U)$ for all $k\in\N\cup\{0\}$ as $\varepsilon\to 0$. Now, 
\begin{align}
\notag 
 \frac{\delta F_Y^{\eps}}{\delta u}(y,u)
 &= 
 -\EV\Big[
   H_{\eps}'\Big( y-h(u,X) \Big)\frac{\delta h}{\delta u}(u, X)
 \Big]
 \\
 \label{equ2}
 &= 
 -\EV\Big[ 
  \rho_{\eps}\Big(
   y - h(u,X)
  \Big) \frac{\delta h}{\delta u}(u, X)
 \Big] \\
 \notag 
 &= 
 -\EV\Big[ 
  \rho_{\eps}\Big(
   y - h(u,X)
  \Big) 
  \EV\Big[ \frac{\delta h}{\delta u}(u, .)\Big| h(u,X) \Big]
 \Big]
 \\ 
 \notag 
 &=
 -
 \int_{\mathbb{R}} 
 \rho_{\eps}\Big(
   y - z
  \Big) 
  \EV_{(u,z)}\Big[ \frac{\delta h}{\delta u}(u, .)\Big]
  f_Y(z,u)\,dz
  \\
  \notag 
  &= 
  \Big(
  \rho_{\eps} \ast
  \EV_{(u,.)}\Big[ \frac{\delta h}{\delta u}\Big]
  f_Y(.,u)
  \Big)(y)
\end{align}
implies that 
\begin{equation}
 \label{e:magic1}
 \frac{\delta F_Y}{\delta u}(y,u) 
 = -f_Y(y,u) \, \EV_{(u,y)}\Big[ \frac{\delta h}{\delta u}\Big]
\end{equation}
whence \eqref{e:impl_der} yields the gradient formula \eqref{e:grad}.

Ad \eqref{e:Hess_q}. 
Let $u_i$ be the $i$-th component of $u$. Differentiating \eqref{e:impl_der} again yields, componentwise,  
\begin{equation}
\label{e:impl_der2}
 \Big(
  \frac{\del f_Y}{\del y}\frac{\delta q_{\alpha}}{\delta u_j}
  +
  \frac{\delta f_Y}{\delta u_j}
 \Big)\frac{\delta q_{\alpha}}{\delta u_i}
 + f_Y \frac{\delta^2 q_{\alpha}}{\delta u_j \delta u_i}
 + \frac{\delta f_Y}{\delta u_i}\frac{\delta q_{\alpha}}{\delta u_j}
 + \frac{\delta^2 F_Y}{\delta u_j \delta u_i}
 = 0.
\end{equation}
To evaluate the last term we differentiate again the mollified version \eqref{equ2}: 
\begin{align}
 \frac{\delta^2 F_Y^{\eps}}{\delta u_j \delta u_i}
 &= 
 -\EV\Big[
  -\rho_{\eps}'\Big(
   y - h(u,X)
  \Big)\frac{\delta h}{\delta u_j}\frac{\delta h}{\delta u_i} 
  +
  \rho_{\eps}\Big(
   y - h(u,X)
  \Big)\frac{\delta^2 h}{\delta u_j \delta u_i}
  \Big]  
  \\
  \notag 
  &= 
  \EV\Big[
  \rho_{\eps}'\Big(
   y - h(u,X)
  \Big)
  \EV\Big[ \frac{\delta h}{\delta u_j}\frac{\delta h}{\delta u_i}
    \Big| h(u,X) \Big]
  \Big]
  -
  \EV\Big[
  \rho_{\eps}\Big(
   y - h(u,X)
  \Big)
  \EV\Big[ \frac{\delta^2 h}{\delta u_j \delta u_i} 
    \Big| h(u,X) \Big]
  \Big]  
  \\
  \notag 
  &= 
  \Big( 
   \rho_{\eps}'\ast  
   \EV_{(u,.)}\Big[ \frac{\delta h}{\delta u_j}\frac{\delta h}{\delta u_i}
    \Big] f_Y(.,u)
  \Big) (y)
  -
  \Big(
   \rho_{\eps} \ast  
   \EV_{(u,.)}\Big[ \frac{\delta^2 h}{\delta u_j \delta u_i}
    \Big] f_Y(.,u)
  \Big) (y)
\end{align}
and convergence in $C^k_{\textup{loc}}(\R\times\U)$ implies
\begin{equation}
 \label{e:magic2}
 \frac{\delta^2 F_Y}{\delta u_j \delta u_i}
 = 
 \Big(\del_y f_Y\Big) \, 
 \EV_{(u,y)}\Big[ 
    \frac{\delta h}{\delta u_j}\frac{\delta h}{\delta u_i}
 \Big]
 + 
 f_Y\,
 \del_y 
 \EV_{(u,y)}\Big[ 
    \frac{\delta h}{\delta u_j}\frac{\delta h}{\delta u_i}
 \Big]
 - f_Y\,  
   \EV_{(u,y)}\Big[ \frac{\delta^2 h}{\delta u_j \delta u_i}
    \Big]
\end{equation}
Now we note that 
$ 
\del_y \cov_{(u,y)}[ 
  \frac{\delta h}{\delta u_j}, \frac{\delta h}{\delta u_i}
 ] 
 = 
  \del_y \EV_{(u,y)}[ 
  \frac{\delta h}{\delta u_j} \frac{\delta h}{\delta u_i} ] 
 -
  \del_y 
  (\EV_{(u,y)}[ \frac{\delta h}{\delta u_j} ]
  \EV_{(u,y)}[ \frac{\delta h}{\delta u_i} ] 
 )
$.
From \eqref{e:impl_der2} with \eqref{e:grad}, \eqref{e:magic1} and \eqref{e:magic2} we thus get
\begin{align} 
 -f_Y\frac{\delta^2 q_{\alpha}}{\delta u_j \delta u_i}
 &= 
 \Big( 
  (\del_y f_Y) \EV_{(u,y)}\Big[ \frac{\delta h}{\delta u_j} \Big]
  + 
  \frac{\delta f_Y}{\delta u_j}
 \Big) \EV_{(u,y)}\Big[ \frac{\delta h}{\delta u_i} \Big] 
 + 
 \frac{\delta f_Y}{\delta u_i}\EV_{(u,y)}\Big[ \frac{\delta h}{\delta u_j} \Big]
 \\
 \notag 
 &\phantom{==}
 +
 (\del_y f_Y)\Big( 
  \cov_{(u,y)}\Big[\frac{\delta h}{\delta u_j}, \frac{\delta h}{\delta u_i} \Big]
  + 
  \EV_{(u,y)}\Big[ \frac{\delta h}{\delta u_j} \Big] 
  \EV_{(u,y)}\Big[ \frac{\delta h}{\delta u_i} \Big] 
 \Big)
  \\
 \notag 
 &\phantom{==}
 +
 f_Y\Big( 
  \del_y\cov_{(u,y)}\Big[\frac{\delta h}{\delta u_j}, \frac{\delta h}{\delta u_i} \Big]
  + 
  \del_y\Big( \EV_{(u,y)}\Big[ \frac{\delta h}{\delta u_j} \Big] 
  \EV_{(u,y)}\Big[ \frac{\delta h}{\delta u_i} \Big] 
  \Big)
 \Big)
   \\
 \notag 
 &\phantom{==}
 -
 f_Y\,
   \EV_{(u,y)}\Big[ \frac{\delta^2 h}{\delta u_j \delta u_i} \Big]  
 \\
 \notag 
 &=
  (\del_y f_Y) \,
  \cov_{(u,y)}\Big[\frac{\delta h}{\delta u_j}, \frac{\delta h}{\delta u_i} \Big]
  + 
  f_Y \,
  \del_y\cov_{(u,y)}\Big[\frac{\delta h}{\delta u_j}, \frac{\delta h}{\delta u_i} \Big]
   -
 f_Y\,
   \EV_{(u,y)}\Big[ \frac{\delta^2 h}{\delta u_j \delta u_i} \Big]  
\end{align}
as desired. 
\end{proof}

We will now drop $u$ from the notation and simply write $S_y$, $\EV_y$ etc.\ with the implicit understanding that $u$-dependency is preserved.

\subsection{Surface variation along normal vector}\label{sec:surf_var}
Assume that $\mathcal{C} = \R^n$. Then we can vary the surface $S_{q_{\alpha}(u)}$ infinitesimally along its normal direction $\nabla h$, and the surface variation formula \eqref{e:surf_var} yields 
\begin{align}
 \frac{\del}{\del y}\Big|_{y = q_{\alpha}(u)}\,
 \Big(
 f_Y(y,u)\var_y\Big[
  \frac{\delta h}{\delta u}
 \Big] \Big) 
 &= 
 \frac{\del}{\del y}\Big|_{y = q_{\alpha}(u)}\,\int_{S_y}
 \Big(\frac{\delta h}{\delta u}-\EV_y\Big[\frac{\delta h}{\delta u}\Big]\Big)^2\,\frac{f_X}{|\nabla h|}\,d\sigma_y
 \\
 \notag 
 &= 
 \int_{S_y}
 T\Big[\Big(\frac{\delta h}{\delta u}-\EV_y\Big[\frac{\delta h}{\delta u}\Big]\Big)^2\,\frac{f_X}{|\nabla h|}\Big]\,d\sigma_y\, \Big|_{y = q_{\alpha}(u)}
 \\ 
 \notag 
 &=
 f_Y(q_{\alpha}(u),u)\,
 \EV_{q_{\alpha}(u)}\Big[
  f_X^{-1}\,\textup{div}\,\Big(
   \Big(\frac{\delta h}{\delta u}-\EV_{q_{\alpha}(u)}\Big[\frac{\delta h}{\delta u}\Big]\Big)^2
   \frac{f_X}{|\nabla h|}\nu 
  \Big)
 \Big]
\end{align}
where $\nu = (\nabla h)/|\nabla h|$ and 
we agree that $\gamma^2 = \gamma\otimes\gamma$ for $\gamma = \delta h/\delta u \in \R^d$.

Hence the Hessian \eqref{e:Hess_q} can be expressed as 
\begin{align}
 \label{e:Hess_q_2}
 \frac{\delta^2 q_{\alpha}}{\delta u^2}
 = 
 \EV_{q_{\alpha}(u)}\Big[&
  \frac{\delta^2 h}{\delta u^2} 
  - 
  f_X^{-1}\,\textup{div}\,\Big(
  \Big(\frac{\delta h}{\delta u}-\EV_{q_{\alpha}(u)}\Big[\frac{\delta h}{\delta u}\Big]\Big)^2
   \frac{f_X}{|\nabla h|}\nu 
  \Big)
 \Big] \\
 \notag 
 =
 \EV_{q_{\alpha}(u)}\Big[& 
   \frac{\delta^2 h}{\delta u^2} 
  - 
  \Big\langle
   \nabla\Big(
  \Big(\frac{\delta h}{\delta u}-\EV_{q_{\alpha}(u)}\Big[\frac{\delta h}{\delta u}\Big]\Big)^2
   \Big),
   \frac{1}{|\nabla h|}\nu 
  \Big\rangle 
  - 
  f_X^{-1} 
  \Big(\frac{\delta h}{\delta u}-\EV_{q_{\alpha}(u)}\Big[\frac{\delta h}{\delta u}\Big]\Big)^2
  \Big\langle \nabla\frac{f_X}{|\nabla h|}, \nu \Big\rangle \\
  \notag 
  &- 
  \Big(\frac{\delta h}{\delta u}-\EV_{q_{\alpha}(u)}\Big[\frac{\delta h}{\delta u}\Big]\Big)^2
  \frac{n-1}{|\nabla h|}\kappa_{q_{\alpha}(u)} 
  \Big] 
\end{align}
where $\kappa_{q_{\alpha}(u)}$ is the mean curvature ~\eqref{e:mean_curv} on $S_{q_{\alpha}(u)}$. The Hessian thus decomposes into the expected second order contribution $\EV_{q_{\alpha}(u)}[\delta^2 h/(\delta u)^2]$, two transport terms and a mean curvature correction.

\section{Convexity properties for elliptical $X$}\label{sec:ell_X}
Let us assume that $X$ is elliptically distributed in $\mathcal{C} = \R^n$ and suppose that $h: \R^n\times\R^n\to\R$ is of the form $h(u,x) = \vv<u,x>$. 

Since elliptical random variables are affine transformations of spherical variables, we have $X = M P + w$ for a matrix $M$, a vector $w$, and $P$ has density
\begin{equation}
 f_{P}(x) = \frac{1}{Z}\psi\Big( \vv<x, x> \Big) 
\end{equation}
where $\psi>0$, $\psi'(z)<0$ for $z\neq0$ and $Z$ is a normalization constant.

Thus the pay-off function can be represented as 
$h(u,X) 
= h(u, M P + w) 
= \vv<M^{\top}u, P> + \vv<u, w> 
= h(\tilde{u}, P) + \vv<u,w>$ where $\tilde{u} = M^{\top}u$. This yields 
$q_{\alpha}(u) 
= \tilde{q}_{\alpha}(\tilde{u}) + \vv<u,w>$ where $\tilde{q}_{\alpha}(\tilde{u})$ is the $\alpha$-quantile of $h(\tilde{u}, P)$. 
Therefore, convexity of $\tilde{q}_{\alpha}$ at $\tilde{u}$ is equivalent to convexity of $q_{\alpha}$ at $u$. 

Suppose $\tilde{u}\neq 0$. Then $\nabla \tilde{h} = \tilde{u}$, $ \nu = \tilde{u}/|\tilde{u}| $ and equation~\eqref{e:Hess_q_2} reduces to 
\begin{align}
\label{e:hess_ell}
 \frac{\delta^2 \tilde{q}_{\alpha}}{\delta \tilde{u}^2}
 &= 
 -\frac{1}{|\tilde{u}|}\EV_{\tilde{q}_{\alpha}(\tilde{u})}
 \Big[ 
  \Big(P-\EV_{\tilde{q}_{\alpha}(\tilde{u})}[P]\Big)^2\vv<\nu , \nabla\log f_P>
 \Big] \\ 
 \notag 
 &= 
 -\frac{1}{|\tilde{u}|^3}
 f_Y(y,u)^{-1}\int_{S_y}
  \Big(x-\EV_{\tilde{q}_{\alpha}(\tilde{u})}[P]\Big)^2 
  \psi'\Big( \vv<x, x > \Big)
  \Big\langle \tilde{u}, 2 x \Big\rangle 
  \, d\sigma_y(x)
\end{align} 
where $y=\tilde{q}_{\alpha}(\tilde{u})$. Now we parametrize the hypersurface $S_y$ as 
$S_y 
= \{ y\frac{\tilde{u}}{|\tilde{u}|^2} + v: \vv<v,\tilde{u}> = 0 \}$ 
and note that the only term in \eqref{e:hess_ell} which has an indefinite sign thus becomes
\begin{align}
 \langle \tilde{u}, x \rangle 
 &= 
 y.
\end{align}

This implies that $\frac{\delta^2 \tilde{q}_{\alpha}}{\delta \tilde{u}^2}$ is negative definite for $y<0$ and positive definite for $y>0$. 
That is, at $u\neq0$, $q_{\alpha}$ is concave  for $\alpha\to0$ and convex as $\alpha\to1$. 

\begin{remark}
This observation is not new, but a manifestation of the well-known fact that Value-at-Risk is sub-additive in the right tail ($\alpha\to1$) of linear combinations of elliptically distributed risk factors (see \cite[Theorem~6.8]{McNeilFreyEmbrechts2015}).  
\end{remark}

\section{Convexity properties for lognormal $X$}\label{sec:log_X}
Let $h(x,u) = \vv<x,u>$ be linear in $x$ and $u = (u_i)$ with $u_i > 0$ for $i=1,\dots,m+1$. 
Assume that $X$ is multivariate log-normal with density $f_X = \exp(-\phi_X)$, specified by 
\begin{equation}\label{e:log_dens}
 \phi_X(x) 
 = \frac{1}{2}\vv<\log(x)-\mu , A^{-1}(\log(x)-\mu)>
  + \sum_{j=1}^{m+1}\log(x_j) + \frac{m+1}{2}\log(2\pi) + \frac{1}{2}\log\det A
\end{equation}
where $A$ is a symmetric and positive-definite matrix with inverse $B=A^{-1}$, $\mu\in\R^{m+1}$ and $\log(x) = \sum_{j=1}^{m+1}\log(x_j)e_j$ where $e_j$ is the standard basis. That is, $X_i = \exp(Z_i)$ with $Z\sim\mathcal{N}_{m+1}(\mu, A)$.  

This distribution does not fit the set-up of Section~\ref{sec:surf_var} since $f_X  = \exp(-\phi_X)$ is defined only on the cone $\mathcal{C} := \{(x_i > 0\}$, so we cannot apply formula~\eqref{e:hess_ell}.  

Via \eqref{e:cond} the  Hessian formula \eqref{e:Hess_q} is equivalent to an integral over a moving hyperplane. To evaluate this integral, we attach a moving frame of reference to the moving surface and  study the infinitesimal variation of the integrand as seen from the surface. 

The control variable $u$ is kept fixed throughout this section and we will write $S_y(u) = S_y$.

\subsection{Level set geometry in the  body frame of reference}
Consider the scaling diffeomorphism $s_y: S_1\to S_y$, $x\mapsto yx$.  

For a function $\xi: \mathcal{C}\to\R$, the transformation rule yields
\begin{align}
 &\int_{S_y}\xi\exp(-\phi_X)\,\textup{vol}^{S_y} 
 =  \int_{S_1} s_y^*\xi\exp(-s_y^*\phi_X)\,s_y^*\textup{vol}^{S_y} \\ 
 \notag 
 &= \int_{S_1}\xi(yx)
 \exp\Big(
  -\log(y)^2\beta/2 + \log(y)\sum_j(B\mu)_j 
  \\
  \notag 
  &\phantom{XXXXXX}
  - (m+1)\log(y)
  -\log(y)\sum_{j,i}B_{ji}\log(x_i) - \phi_X(x)
 \Big)\,y^m\textup{vol}^{S_1}_x  \\
 \notag
 &= 
 \exp\Big(
  -\frac{\beta}{2}\log(y)^2 + \log(y)(\sum_j(B\mu)_j-1)
 \Big)
 \int_{S_1}\xi(yx)\exp\Big(-w_y(x) \Big)
 \,\textup{vol}^{S_1}_x  
\end{align}
where $\beta = \sum_{i=1}^{m+1}b_i$, $b_i = \sum_{j=1}^{m+1}B_{ji}$ and
\begin{align}
\label{e:def_gy}
 g_y(x) := \exp(-w_y(x)), \qquad
 w_y(x) := \log(y)a(x) + \phi_X(x), \qquad
 a(x) := \sum_{i=1}^{m+1}b_i\log(x_i) .
\end{align}
Notice that $w_y: \textup{int}\,S_1\to\R$ is smooth where $\textup{int}\,S_1 = S_1\setminus\del S_1$ is the interior.  Moreover, $g_y = \exp(-w_y)$ can be continuously extended to $S_1$ with $g_y|\del S_1 = 0$.

Let $Z_y := \int_{S_1}\exp(-w_y)\,\vol^{S_1}$ and 
\begin{equation}
 \mu_y = Z_y^{-1}\exp(-w_y)\,\textup{vol}^{S_1}
\end{equation}
the associated probability measure on $S_1$. From the above scaling transformation we obtain $\EV_y[X_i] = y\EV_{\mu_y}[X_i]$ where $X_i = \vv<X,e_i>$ for the $i$-th standard basis vector. 

As above, consider $\gamma(x) = \vv<v,x>$ for a fixed vector $v\in\R^{m+1}$. The variance from equation~\eqref{e:Hess_q} can be now expressed as 
\begin{align}
 \vv<v ,\var_y[\frac{\delta h}{\delta u}]v >
 = y^2\var_{\mu_y}[\gamma]. 
\end{align}
This gives
\begin{align}
 \vv<v , f_Y(y)\var_y[\frac{\delta h}{\delta u}]v >
 = 
 \exp\Big(
   -\frac{\beta}{2}\log(y)^2 + \log(y)(\sum_j(B\mu)_j+1)
 \Big)
 Z_y\var_{\mu_y}[\gamma]. 
\end{align}
Therefore, 
\begin{align}\label{e:del_yVar}
 \del_y\vv<v , f_Y(y)\var_y[\frac{\delta h}{\delta u}]v >
 = 
 \frac{1}{y}&\Big(\Big(
  -\beta\log(y) + \sum_j(B\mu)_j+1
 \Big)\var_{\mu_y}[\gamma]
 -
 \EV_{\mu_y}\Big[a(\gamma - \EV_{\mu_y}[\gamma])^2 \Big]
 \Big) \\
 \notag 
 &\cdot\exp\Big(
   -\frac{\beta}{2}\log(y)^2 + \log(y)(\sum_j(B\mu)_j+1)
 \Big)
 Z_y 
\end{align}
and this expression is negative if, and only if, 
\begin{align}
\label{e:negCond}
 \Big(
  -\beta\log(y) + \sum_j(B\mu)_j+1 - \EV_{\mu_y}[a] 
 \Big)\var_{\mu_y}[\gamma]
 \le 
 \textup{Cov}_{\mu_y}\Big[a; (\gamma - \EV_{\mu_y}[\gamma])^2\Big] 
 = -y\del_y\var_{\mu_y}[\gamma]
\end{align}

\begin{remark}
The validity of this inequality (once shown) implies a Gronwall estimate for the variance $\var_{\mu_y}[\gamma]$. However, this is not very useful since we will show below in \eqref{e:COV} and \eqref{e:LHS_cond} that the right hand side of this equation is positive while the left hand side is negative.  
\end{remark}

\begin{remark}[Score] 
For a smooth function $\xi: \R_{>0}\times S^1\to\R$ we have $\del_y\EV_{\mu_y}[\xi] = \EV_{\mu_y}[\del_y\xi] - y^{-1} \cov_{\mu_y}[a, \xi]$. This means that $-y^{-1}a$ is the score associated to $(\mu_y)_{y>0}$ in the terminology of information theory.   
\end{remark}

We parametrize $S_1 \cong \{ x\in\R^m: x_i\ge0, \sum_{i=1}^m u_i x_i \le 1 \} = \mathcal{S}_1$ via $x_{m+1} = u_{m+1}^{-1}(1-\sum_{i=1}^m u_i x_i)$, and thus define $g_y^m(x_1,\dots,x_m) = g_y(x_1,\dots,x_m, x_{m+1})$. 

Critical points of $g_y^m$ are characterized by
\begin{equation}
 \log(y)\Big(\frac{b_i}{x_i} - \frac{u_i}{u_{m+1}}\frac{b_{m+1}}{x_{m+1}}\Big)
 = 
 -\del_i \phi_X(x) + \frac{u_i}{u_{m+1}}\del_{m+1}\phi_X
\end{equation}
for all $i = 1,\dots,m$ and where $x_{m+1}$ is viewed as a dependent variable. From the definition of $\phi_X$ this is equivalent to
\begin{align}
\label{e:CPequ_m}
 \log(y)\Big(b_i u_{m+1}x_{m+1} - b_{m+1} u_ix_i\Big) 
 &=
 \Big(
  (B\mu)_i-1
 \Big)u_{m+1}x_{m+1}
 -  \Big(
  (B\mu)_{m+1}-1
 \Big)u_i x_i \\ 
 \notag 
 &\phantom{XX}
 - \sum_{j=1}^{m+1}\Big(
  B_{ij}u_{m+1} x_{m+1} -  B_{m+1,j}u_i x_i
 \Big)
 \log(x_j)
\end{align}

\begin{proposition}[Barycentric solution]\label{prop:bc_sol}
Assume that $b_i>0$ for all $i=1,\dots,m+1$. Then there are $y_0$ and $y_1$ such that there exists a solution $x^*(y)$ to  \eqref{e:CPequ_m} for all $0<y\le y_0$ and all $y_1\le y<\infty$. Moreover, $x^*(y)$ depends smoothly on $y$, and
\[
 \lim_{y\to0} x^*(y)
 = x^*
 = \lim_{y\to\infty} x^*(y),
 \qquad
 x^* = \Big( \frac{b_i}{u_i\beta} \Big)_{i=1}^{m+1}.
\]
\end{proposition}

\begin{proof}
Let ${\tau}>0$ such that ${\tau} < \frac{b_i}{u_i\beta}$ for all $i=1,\dots,m+1$. Define $\mathcal{S}_1^{\tau} = \{x \in \mathcal{S}_1: x_i \ge {\tau}\}$. Let $\rho: \mathcal{S}_1^{\tau}\to \R^m$ denote the function defined by the right hand side of \eqref{e:CPequ_m}. This function is smooth and bounded by construction.  Thus \eqref{e:CPequ_m} is equivalent to 
\begin{equation}
 b_i - \sum_{j=1}^m M_{ij} x_j
 =
 b_i\Big(1-\sum_{k=1}^mu_k x_k\Big) - b_{m+1}u_i x_i
 = \frac{\rho_i(x)}{\log(y)}
\end{equation}
where the invertible matrix $M$ is defined by the first equation.
Hence \eqref{e:CPequ_m} is equivalent to the fixed point problem $x=\Psi_y(x)$ with parameter $y$ where 
\begin{equation}
 \Psi_y(x_{(m)}) = M^{-1} b_{(m)} - \frac{1}{\log(y)}M^{-1}\rho(x_{(m)})
\end{equation}
with $b_{(m)} = (b_1,\dots,b_m)$ and $x_{(m)}\in\mathcal{S}_1^{\tau}$.  By choosing $y_0$ sufficiently small (resp.\ $y_1$ sufficiently large), it follows that the non-constant part of $\Psi_y$ tends to $0$ uniformly in $x_{(m)}\in\mathcal{S}_1^{\tau}$. This implies that $\Psi_y: \mathcal{S}_1^m\to\mathcal{S}_1^m$ is a contraction with uniform and smooth dependency on $y$. By the uniform contraction principle there is therefore a solution $x_{(m)}(y) = \Psi_y(x_{(m)}(y))$, and the solution depends smoothly on $y$.  

Moreover, $M^{-1}b_{(m)} = (\frac{b_i}{u_i\beta})_{i=1}^{m}$ and since $u_{m+1}x^*_{m+1} = 1 - \sum_{k=1}^m b_k/\beta = b_{m+1}/\beta$ the claim follows.
\end{proof}

\begin{remark}
Since $\beta>0$ by positive definiteness of $B$, at least one $b_i$ has to satisfy $b_i>0$. If there is now a $b_j\le0$ then the construction in the proof above shows that the assumption of a fixed ${\tau}>0$ and a critical point solution $x_{(m)}(y)\in\mathcal{S}_1^{\tau}$ for $y\to0$ or $y\to\infty$ leads to a contradiction. Hence when there is at least one $b_j\le0$ critical point solutions are asymptotically confined to the boundary layer of $S_1$. 
\end{remark}

Characterizing the boundary layer solutions of \eqref{e:CPequ_m} involves many case distinctions. To avoid this, we will confine ourselves to the case $m=1$. We assume the ordering of indices is such that $b_1 > 0$. Equation~\eqref{e:CPequ_m} thus specializes to
\begin{align}\label{e:CPequ2}
 \log(y)\Big(b_1 - \beta u_1x_1\Big) 
 =
  (&B\mu)_1 - 1 + B_{12}\log(u_2) 
 - \Big( (B\mu)_1 + (B\mu)_2 - 2 + b_2\log(u_2) \Big)u_1x_1 \\
 \notag 
 & - \Big(B_{11} - b_1u_1x_1 \Big) \log(x_1) 
 - \Big(B_{12} - b_2u_1x_1 \Big) \log(1-u_1x_1)
\end{align}

\begin{proposition}(Corner solutions)\label{prop:corn_sol}
Let $m=1$ and $b_1>0$. The following are true. 
\begin{enumerate}
\item 
There is a $y_2>0$ and a critical point solution $x_1^*(y)$ of \eqref{e:CPequ2} which depends smoothly on $y>y_2$. Moreover, $x_1^*(y)\to0$ as $y\to\infty$. 
\item 
If $b_2>0$, then there is a $y_1>0$ and a critical point solution $x_1^*(y)$ of \eqref{e:CPequ2} which depends smoothly on $y>y_1$. Moreover, $x_1^*(y)\to u_1^{-1}$ as $y\to\infty$. 
\item 
If $b_2 < 0$, then there is a $y_0>0$ and a critical point solution $x_1^*(y)$ of \eqref{e:CPequ2} which depends smoothly on $y<y_0$. Moreover, $x_1^*(y)\to u_1^{-1}$ as $y\to 0$. 
\end{enumerate}
\end{proposition}

\begin{proof}
Item (1). 
Let $u_1^{-1} > {\tau} > 0$ and define $\mathcal{S}_1^{\tau} = \{x\in\mathcal{S}_1: x\le {\tau}\}$. 
We transform \eqref{e:CPequ2} to a fixed point equation $x = \Psi_y(x)$ for $x\in\mathcal{S}_1^{\tau}$, where $\Psi_y(x) = \exp(M_y(x))$ and
\begin{align}
 M_y(x)
 = 
  \Big( B_{11} - b_1 u_1 x \Big)^{-1}
  \Big(&
    (B\mu)_1 - 1 + B_{12}\log(u_2)
    - \Big( (B\mu)_1 + (B\mu)_2 - 2 + b_2\log(u_2) \Big) u_1 x_1 \\
    \notag
    &- \Big(B_{12} - b_2 u_1 x_1 \Big) \log(1 - u_1 x_1)
      - \log(y)\Big(b_1 - \beta u_1 x_1\Big)  
  \Big)
\end{align}
Since $|\Psi_y(z_2) - \Psi_y(z_1)| \le m(y) |z_2-z_1|$ with $m(y) = \exp(\max_{z\in\mathcal{S}_1^{\tau}}M_y(z))|\max_{z\in\mathcal{S}_1^{\tau}}M_y'(z)| \to 0$ as $y\to\infty$ it follows that $\Psi_y: \mathcal{S}_1^{\tau} \to \mathcal{S}_1^{\tau}$ is a contraction which is uniform and smooth in $y>y_2$ for a suitable $y_2$. Hence a fixed point solution $x^*(y) = \Psi_y(x^*(y))$ exists for $y>y_2$, and  $x^*(y)$ is smooth in $y$.  Moreover, $x^*(y) = \Psi_y(x^*(y))\to 0$ as $y\to\infty$. 

Items (2) and (3) follow analogously by restricting to $\mathcal{S}_1^{\tau} = \{x\in\mathcal{S}_1: x \ge {\tau}\}$.  
\end{proof}

As $y\to\infty$, this implies the zeroth and first order asymptotic expansions for the critical point solution $x^*_1(y)$ in the left corner $\{x\in\mathcal{S}_1: x\le {\tau}\}$, respectively:
\begin{align}
\label{e:asym_log0}
 \log x^*_1(y)
 &\sim 
   B_{11}^{-1}
    \Big(
     c_1
     - b_1\log(y)
    \Big) 
 \\
 \label{e:asym_log1}
  \log x^*_1(y)
 &\sim 
   B_{11}^{-1}
    \Big(
      1 
      + 
      B_{11}^{-1} b_1 u_1 x^*_1(y) \Big) 
    \Big(
     c_1
     - b_1\log(y)
    \Big) \\
    \notag 
    &\phantom{XX}
    - B_{11}^{-1}
     \Big( (B\mu)_1 + (B\mu)_2 - 2 
      + B_{12}   
          + \beta\log(y)\Big) u_1 x^*_1(y) 
\end{align}
where $c_1 = (B\mu)_1 - 1 + B_{12}\log(u_2)$.

Similarly, the zeroth order expansion of $x^*_1(y)$ in the right corner $\{x_1\in\mathcal{S}_1: x_1\ge {\tau}\}$ is as follows. Let either $b_2>0$ and $y\to\infty$ or $b_2<0$ and $y\to0$. Then
\begin{equation}
\label{e:asym_RC}
 1 - u_1 x_1^*(y)
 \sim 
 \exp\Big(
  B_{22}^{-1}\Big(
   c_2 - b_2 \log(y)
  \Big)
 \Big)
 = 
 u_2 \exp\Big(
  B_{22}^{-1}\Big(
     (B\mu)_2 - 1 - b_2 \log(y)
  \Big)
 \Big)
\end{equation}
where $c_2 = (B\mu)_2 - 1 + B_{22}\log(u_2)$.

\subsection{Concavity in the left tail}

\begin{theorem}[Left tail]\label{thm:conc}
The following are true. 
\begin{enumerate}
\item 
Assume $b_i > 0$ for all $i = 1,\dots,m$.
If $q_{\alpha_0}(u) < \eta\exp\Big(\frac{1}{\beta}(\sum_j(B\mu)_j+1)\Big)$ for some $0 < \eta < \min\{u_1, \dots,u_{m+1}\}$,  then $q_{\alpha}$ is strictly concave at $u$ for all $\alpha\le\alpha_0$. 
\item 
Let $m=1$ and $b_2 < 0$. Then there exists an $\alpha_0 > 0$ such that $q_{\alpha}$ is concave in $u$ for all $\alpha\le\alpha_0$. 
\end{enumerate}
\end{theorem}

\begin{proof}
Ad Item (1).
Fix $0 < \eta < \min\{u_1, \dots,u_{m+1}\}$ and consider $\tilde{u} = u/\eta$. Since $h(x,\tilde{u}) = 1$ implies that $x_i<1$ it follows that $a|S_1(\tilde{u}) < 0$. By equation~\eqref{e:del_yVar} it follows that $q_{\alpha}$ is strictly concave at $\tilde{u}$ if $q_{\alpha_0}(\tilde{u}) < \exp\Big(\frac{1}{\beta}(\sum_j(B\mu)_j+1)\Big)$. Since $q_{\alpha_0}(\tilde{u}) = \eta^{-1}q_{\alpha}(u)$ the claim follows. 

Ad Item (2). See Section~\ref{sec:b_2_le_0}.
\end{proof}

\begin{remark}
The case $b_2=0$ is not covered by Theorem~\ref{thm:conc} and would require further analysis.
\end{remark}

\begin{remark}
Regarding Item (1), note that $\inf h(u,X) = 0< \exp\Big(\frac{1}{\beta}(\sum_j(B\mu)_j+1)\Big)$. Hence for all admissible $u$ there exists a threshold level $\alpha_0 = \alpha_0(u)$ below which $q_{\alpha}$ is concave at $u$. 
\end{remark}

If we interpret $h(u,X) = \sum_j u_j X_j$ as a linear portfolio of lognormally distributed market prices, the associated loss function is $l(u,X) = h(u,X^0 - X)$, where $X^0$ is the current observation of market prices. For the following conclusion the case $m=1$ is sufficient. 

\begin{corollary}[Subadditivity of Value-at-Risk in loss tail]\label{cor:losses_sa}
Consider losses $L_1 = X_1^0 - X_1$ and $L_2 = X_2^0 - X_2$. Then there is confidence level $\alpha_0\in(0,1)$ such that $q_{1-\alpha}$ is concave at $(1-\lambda)e_1 + \lambda e_2$ for all $\lambda\in[0,1]$ and all $\alpha_0\le\alpha<1$. Hence 
\[
 \VaR_{\alpha}[L_1 + L_2] \le  \VaR_{\alpha}[L_1] + \VaR_{\alpha}[L_2]  
\]
i.e., sub-additivity holds. 
\end{corollary}

\begin{proof}
Notice firstly that strict concavity of $q_{1-\alpha}$ extends to concavity at the boundary of $\mathcal{C}$, hence we may consider strategies $U_1 = e_1$ and $U_2 = e_2$. Since the line segment connecting $U_1$ and $U_2$ is compact, we can find $\alpha_0<1$ such that concavity holds along the segment for all $\alpha_0\le\alpha<1$. Hence $q_{1-\alpha}((1-\lambda)U_1 + \lambda U_2) \ge (1-\lambda)q_{1-\alpha}(U_1) + \lambda q_{1-\alpha}(U_2)$ by concavity. With $\lambda = 1/2$ this yields $\VaR_{\alpha}[L_1 + L_2] = X_1^0+X_2^0+ q_{\alpha}(-e_1-e_2) \le X_1^0+X_2^0+ q_{\alpha}(-e_1) + q_{\alpha}(-e_2) = \VaR_{\alpha}[L_1] + \VaR_{\alpha}[L_2]$.  
\end{proof}

\begin{remark}
For linear portfolios of elliptically distributed risk factors it is well-known that Value-at-Risk is sub-additive for $1/2\le\alpha<1$ (see \cite[Theorem~6.8]{McNeilFreyEmbrechts2015}). 
To the best of our knowledge, this corollary is a new instance of the phenomenon of $\VaR$-sub-additivity in the tail of losses. 
\end{remark}

\subsection{Convexity in the right tail ($m=1$)}
To analyze the boundary layer critical point solutions of \eqref{e:CPequ2}, study the mass distribution of $\mu_y$ with varying $y$.  Figure~\ref{fig:gydens} contains a visualization of this mass flow. 
 
\begin{figure}[htbp] \centering \includegraphics[width=0.47\textwidth]{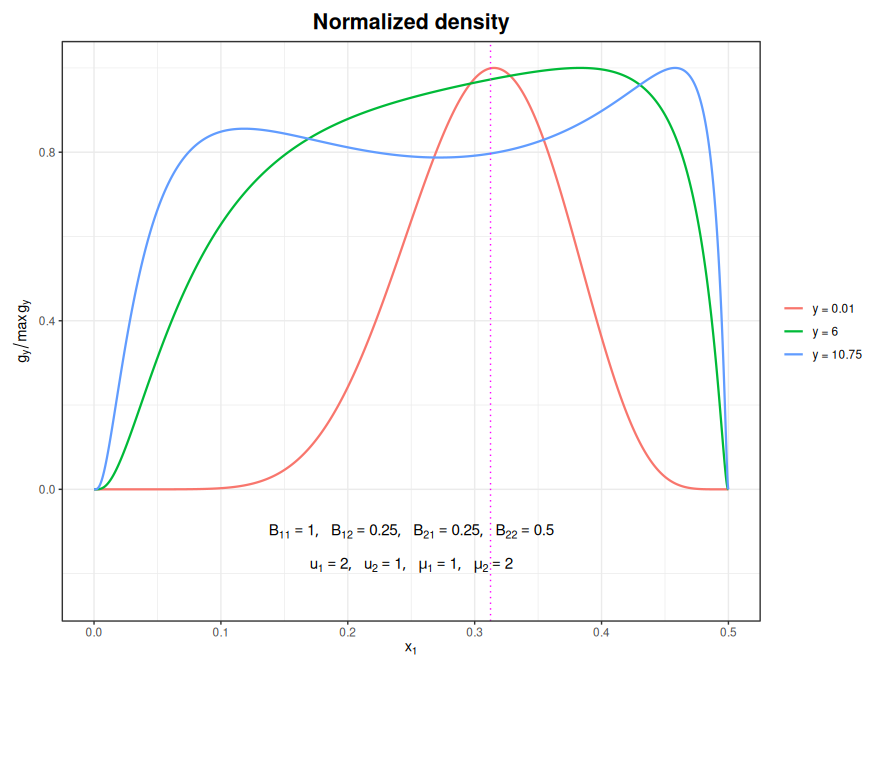} \hfill 
\includegraphics[width=0.47\textwidth]{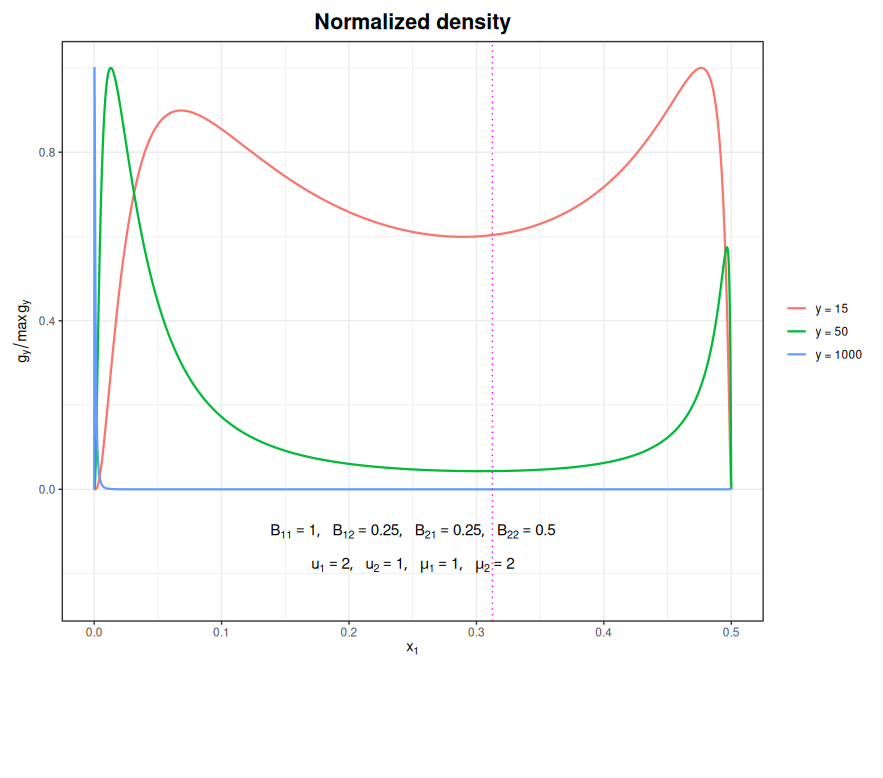}
\vspace{-10mm}
\caption{Normalized densities $g_y/\max(g_y)$ for the indicated values of $y$ where  $10.75 \approx \exp(\beta^{-1}((B\mu)_1+(B\mu)_2+1))$ is the (non-sharp) threshold for concavity from Theorem~\ref{thm:conc}. The initially concentrated peak spreads and wanders to the right, then the mass flows to the left and finally concentrates at the left corner. The central minimum finally sits at the exact location ($x^*$ in Proposition~\ref{prop:bc_sol}, dotted vertical line) of the original peak, and is created by the mass flowing to the left while the local maximum at the right corner is never annihilated completely. 
}
\label{fig:gydens} 
\end{figure}

The curvature at critical point solutions, $x^*_1 = x^*_1(y)$, of \eqref{e:CPequ2} is determined by the sign of $g_y''(x^*_1) = -w''(x^*_1)g_y(x^*_1)$, where 
\begin{align}
\label{e:w_y''}
 w_y''(x_1) 
 &= -\log{y}\Big(
  \frac{b_1}{x_1^2} + \frac{u_1^2}{u_2^2}\frac{b_2}{x_2^2}
 \Big)
  + \frac{\del^2\phi_X}{(\del x_1)^2} 
  - 2\frac{u_1}{u_2}\frac{\del^2\phi_X}{(\del x_1)(\del x_2)}
  + \frac{u_1^2}{u_2^2}\frac{\del^2\phi_X}{(\del x_2)^2} 
  \\ 
  \notag 
  &= 
  -\log{y}\Big(
  \frac{b_1}{x_1^2} + \frac{u_1^2}{u_2^2}\frac{b_2}{x_2^2}
 \Big) 
  +
  \frac{B_{11}}{x_1^2} - 2\frac{u_1}{u_2}\frac{B_{12}}{x_1x_2} + \frac{u_1^2}{u_2^2}\frac{B_{22}}{x_2^2}
  \\
  \notag 
  &\phantom{XX} 
  +
  \frac{1}{x_1^2} \Big(
    (B\mu)_1 - 1
  \Big)
  +
  \frac{u_1^2}{u_2^2 x_2^2} \Big(
    (B\mu)_2 - 1
  \Big)  
  -
   \sum_{j=1}^2 \Big(\frac{1}{x_1^2}  B_{1j} +  
   \frac{u_1^2}{u_2^2 x_2^2} B_{2j} \Big) \log(x_j)  
\end{align}
and we consider $g_y$ and $w_y$ as functions of the free variable, $x_1$. 

\subsubsection{Center solution}
When $b_1>0$ and $b_2>0$, the center solution of \eqref{e:CPequ2} appears asymptotically at both ends of the range of the quantile function. For $y\to0$ and  for $y\to\infty$, it is asymptotically  given by Proposition~\ref{prop:bc_sol}.
this solution, $x^*_1(y)$, remains bounded away from the boundary, whence it follows that the sign of
$g_y''(x^*_1(y)) = -w_y''(x^*_1(y)) g_y(x^*_1(y))$ is determined by  
\begin{equation}
 -\Big(
  \frac{b_1}{x_1^2} + \frac{u_1^2}{u_2^2}\frac{b_2}{x_2^2}\Big)\log{y}
\end{equation}
Thus the center solution, $x^*_1(y)$, is a local maximum for small $y$ and local minimum for large $y$. We will use the center solution in Section~\ref{sec:univ}.

\subsubsection{Global maximum of $g_y$ as $y\to\infty$} 
Assume $b_1>0$ and $b_2>0$. Then Proposition~\ref{prop:corn_sol} (Item 2) is symmetric in $x_1$ and $x_2$. Therefore, for $i=1,2$, we can use the zeroth order asymptotic expansion~\eqref{e:asym_log0} to evaluate $w_y''$ at $x^*_i(y)$. Let us now consider the case $i=1$. 

Let $x_y^*$ be the critical point solution in the corner $x_1=0$ as $y\to0$. This solution is asymptotically given by \eqref{e:asym_log0}. We use equation~\eqref{e:w_y''} to evaluate 
\begin{align}
 -w_y''(x_y^*)
 &\sim 
 \Big(x_y^*\Big)^{-2}\Big(
  b_1\log(y) + \sum_{j=1}^2B_{1j}\log(x_j) - (B\mu)_1 + 1
  - B_{11}
 \Big) \\ 
 \notag 
 &\sim 
  \Big(x_y^*\Big)^{-2}\Big(
  b_1\log(y) 
  + B_{11}\frac{c_1-b_1\log(y)}{B_{11}}
  + B_{12}\log(\frac{1-u_1 x_1}{u_2})
  - (B\mu)_1 + 1
  - B_{11}
 \Big) \\
 \notag 
 &= -\frac{B_{11}}{(x^*_y)^2}\\
 \notag
 &\sim 
 - B_{11}\exp\Big( -2\frac{c_1-b_1\log(y)}{B_{11}} \Big)
\end{align}
asymptotically as $y\to\infty$, and where we keep only terms of order $(x_y^*)^{-2}\log(y)$ and $(x_y^*)^{-2}$. 

The case $i=2$ is analogous, and therefore, as $y\to\infty$, the two local critical points are strict local maxima. 
The global maximum is detected by using \eqref{e:asym_log0} again. This yields the global maximum condition 
\begin{equation}
\label{e:GMC}
 \frac{b_1^2}{B_{11}} > \frac{b_2^2}{B_{22}}
\end{equation}
which holds if, and only if, the global maximum approaches the left corner, i.e.\ $x^*_1(y)\to0$ as $y\to\infty$. In this case, the corner at $x_1=0$ dominates asymptotically.  And vice versa for the right corner.

\begin{remark}
In terms of the defining matrix $A=B^{-1}$ in \eqref{e:log_dens} condition \eqref{e:GMC} is equivalent to $A_{22}>A_{11}$, i.e.\ the dominant corner is determined by the component with the largest variance. 
\end{remark}
 
If we have equality in \eqref{e:GMC} then the first order expansion \eqref{e:asym_log1} can be used analogously to determine the influence of $u_1/u_2$ on the global maximum. If all parameters are symmetric in $1$ and $2$, then we have two equally dominant corners, i.e.\ two global maxima.  

When $b_2 < 0$ we know from Proposition~\ref{prop:corn_sol} that there is only one critical point solution as $y\to\infty$, and this solution is then the global maximum.

\subsubsection{Laplace principle}  
Assume that $b_1>0$ and $b_2>0$, and that the left corner dominates, i.e.\ condition~\eqref{e:GMC} is satisfied. 

Let $y_0$ such that $x_{y_0}^*\in(0,\delta)$ with $\delta>0$ sufficiently small so that higher order terms in the Taylor approximation of $w_{y_0}$ can be neglected. For $y\ge y_0$ we have 
\begin{align}
 Z_y 
 &\ge 
 \int_0^{\delta}\exp\Big(
  -w_y(x_y^*)
  - \frac{1}{2}w_y''(x_y^*)\Big(x-x_y^*\Big)^2\Big) \,dx 
  \\
  \notag 
  &\ge
  \exp\Big(-w_y(x_y^*)\Big)
  \frac{\sqrt{\pi}}{\sqrt{w_y''(x_y^*)}} \\
  \notag
  &\sim 
  \frac{\sqrt{\pi}}{\sqrt{B_{11}}}
  \exp\Big(
   - w_y(x_y^*) )
   + \log( x_y^* ) 
   \Big)
\end{align}
where we absorb a factor $2$ from the approximating Gaussian integral.  We use \eqref{e:asym_log0} again to evaluate the leading order contributions to the terms in the bracket as
\begin{align}
 -w_y(x_y^*) + \frac{c_1 - b_1\log(y)}{B_{11}}
 &\sim 
 -\Big( b_1 \log(y) - 1)\log(x_y^*) - \frac{B_{11}}{2}\log(x_y^*)^2 
 \sim 
 \frac{b_1^2}{2 B_{11}} \log(y)^2 
\end{align}
Therefore,
\begin{equation}
 \frac{g_y(x)}{Z_y}
 \underset{\sim}{<}
 \exp\Big(
  -w_y(x) - \frac{b_1^2}{2 B_{11}}\log(y)^2 
 \Big)
 \to 
 0
\end{equation}
pointwise in $\delta \le x \le u_{1}^{-1}-\delta$ as $y\to\infty$. Thus we have the following Laplace principle for smooth functions $f: S_1\to\R$ and $\delta>0$: 
\begin{align}
\label{e:LP}
 \int_{0}^{u_1^{-1}}f(x)g_y(x)\,dx
 &\sim 
  \int_{0}^{\delta}f(x)g_y(x)\,dx
  +  \int_{u_1^{-1}-\delta}^{u_1^{-1}}f(x)g_y(x)\,dx 
  \sim 
  \int_{0}^{\delta}f(x)g_y(x)\,dx
\end{align}
as $y\to\infty$, where the second asymptotic equivalence follows from dominance of the left corner.

\subsubsection{Third derivative}
Differentiating \eqref{e:w_y''} with respect to $x_1$ and keeping only terms of order $x_1^{-3}\log(y)$ and $x_1^{-3}$ yields
\begin{align}
\label{e:w_y'''}
 w_y^{'''}(x^*_1(y))
 &\sim  
 \Big(
  2\log(y)\frac{b_1}{x_1^3} - 2\frac{B_{11}}{x_1^3}
   - 2\frac{(B\mu)_1-1}{x_1^3} 
   + 2\frac{B_{11}\log(x_1)+B_{12}\log(x_2)}{x_1^3}
   -\frac{B_{11}}{x_1^3}
 \Big)\Big|_{x_1 = x^*_1(y)}
 \\
 \notag
 &\sim
 -3\frac{B_{11}}{x^*_1(y)^3}
 + 2\frac{-(B\mu)_1+1+c_1}{x^*_1(y)^3}
 - 2\frac{B_{12}\log(u_2)}{x^*_1(y)^3}
 \\
 \notag 
 &=  -3\frac{B_{11}}{x^*_1(y)^3}
\end{align}
in $C_0$ and as $y\to\infty$, and where the second asymptotic similarity uses the zero order expansion~\eqref{e:asym_log0} and $\log(x_2) = -\log(u_2) + \mathcal{O}(x_1)$.

\subsubsection{Covariance with score}
We continue to assume that the global maximum condition~\eqref{e:GMC} holds, and now evaluate the relevant terms in condition~\eqref{e:negCond}. 
To do so  we use the Laplace principle \eqref{e:LP} to expand $w_y$ around $x^*_1(y)$:
\begin{align}
 \EV_{\mu_y}[a]
 &= 
 Z_y^{-1}\int_0^{u_1^{-1}} a(x_1)g_y(x_1)\,dx_1 \\
 \notag 
 &\sim
 Z_y^{-1}\int
  \Big(a(x^*_1(y)) + a'(x^*_1(y))(x_1-x^*_1(y))\Big)
  \\
  \notag 
  &\phantom{XXXXX}
  \exp\Big(
   -w_y(x^*_1(y)) - \frac{1}{2}w_y''(x^*_1(y))(x_1-x^*_1(y))^2
  \Big)\,dx_1 \\
  \notag 
  &\sim a(x^*_1(y))
\end{align}
Regarding the covariance expression on the right hand side of \eqref{e:negCond} we use the definition $\gamma(x) = \vv<v,x> = v_1x_1 + v_2x_2$ and note that $\gamma(x) - \EV_{\mu_y}[\gamma] = (v_1-v_2 u_1/u_2) x_1$. The scalar $v_1-v_2 u_1/u_2$ does not affect the sign of the covariance term and will therefore be omitted from now on. Thus
\begin{align}
\label{e:COV}
 \cov_{\mu_y}\Big[
  a; \Big(X_1 - \EV_{\mu_y}[X_1]\Big)^2 
 \Big]& 
 \sim  \EV_{\mu_y}\Big[
  \Big(a - a(x^*_1(y))\Big) \Big(X_1 - \EV_{\mu_y}[X_1]\Big)^2 
 \Big] 
 \\
 \notag 
 &\sim 
 Z_y^{-1}\int a'(x^*_1(y))(x_1-x^*_1(y))^3
 \exp\Big(
 -w_y(x^*_1(y)) \\
 \notag 
 &\phantom{XXXXX}
 - \frac{1}{2}w_y^{''}(x^*_1(y))(x_1-x^*_1(y))^2 - \frac{1}{6}w_y^{'''}(x^*_1(y))(x_1-x^*_1(y))^3 
 \Big)
 \,dx_1
\end{align}
and this expression is positive since $a'(x^*_1(y)) \sim b_1/x^*_1(y) > 0$ and \eqref{e:w_y'''} implies a strict right skew.

On the other hand, using again \eqref{e:asym_log0} we find for the left hand side of \eqref{e:negCond} that 
\begin{align}
\label{e:LHS_cond}
 -\beta\log(y) + \sum_j(B\mu)_j+1 - \EV_{\mu_y}[a] 
 &\sim 
 -\beta\log(y) - a(x^*_1(y))\\
 \notag
 &\sim \Big(-\beta + \frac{b_1^2}{B_{11}}\Big)\log(y)
 = -\frac{\det B}{B_{11}}\log(y)
 < 0
\end{align}
Hence we may conclude that there exists $y_0>0$ such that $\del_y\var_{\mu_y}[\gamma] \le 0$ and consequently $\del_y(f_Y(y)\var_y)\le0$ for all $y\ge y_0$, as desired, whence we obtain:  

\begin{theorem}[Right tail]\label{thm:conv}
For all $u = (u_1, u_2)$ with $u_1\ge0$ and $u_2\ge0$ there exists a level $\alpha_0\in(0,1)$ such that $q_{\alpha}$ is convex at $u$ for all $1>\alpha\ge\alpha_0$.  
\end{theorem}

\begin{corollary}[Sub-additivity in the gains tail]\label{cor:gains_sa}
There is confidence level $\alpha_0\in(0,1)$ such that $q_{\alpha}$ is convex at $(1-\lambda)e_1 + \lambda e_2$ for all $\lambda\in[0,1]$ and all $\alpha_0\le\alpha<1$. Hence 
\[
 \VaR_{\alpha}[X_1 + X_2] \le  \VaR_{\alpha}[X_1] + \VaR_{\alpha}[X_2]  
\]
i.e., sub-additivity holds. 
\end{corollary}

\begin{proof}
As in Corollary~\ref{cor:losses_sa}.
\end{proof}

\subsection{From dominance at one corner to dominance at the other end}\label{sec:b_2_le_0} 
The matrix $B$ has to be symmetric and positive definite, but now we assume that $b_1>-b_2>0$. From Proposition~\ref{prop:corn_sol} (Items 1 and 3) it follows that there is a unique maximizing critical point solution $x_y^*$ of \eqref{e:CPequ2} near $x=u_1^{-1}$ as $y\to 0$ and unique maximizing critical point solution $x_y^*$ of \eqref{e:CPequ2} near $x=0$ as $y\to\infty$. 

The above analysis for $y\to\infty$ carries through unchanged. For $y\to 0$, however, we get a negative covariance term, since in \eqref{e:COV} the leading contribution is now multiplied by $a'(x^*_1(y)) \sim b_2/x^*_1(y) < 0$, thus inducing a strict left skew. Condition~\eqref{e:negCond} is therefore not satisfied as $y\to\infty$, which implies Item 2 of Theorem~\ref{thm:conc}.

\section{Universality in the tail regime}\label{sec:univ}

Let $h(x,u) = \vv<u,x>$ and $X$ lognormal in $\R^{m+1}$ with density $\phi_X$ as in Section~\ref{sec:log_X}.
This section is concerned with the structure of $q_{\alpha}(u)$ in the tail regime $\alpha\to0$ and $\alpha\to1$.   
 
\subsection{Asymptotics at the barycenter} 
Assume that $b_k>0$ for $k=1,\dots,m+1$.
As above, $S_1$ is an $m$-dimensional hypersurface in $\C\subset\R^{m+1}$ for fixed $u$. Consider $g_y$ defined in \eqref{e:def_gy}. 
We parametrize $S_1 \cong \{ x\in\R^m: x_i\ge0, \sum_{i=1}^m u_i x_i \le 1 \} = \mathcal{S}_1$ via $x_{m+1} = u_{m+1}^{-1}(1-\sum_{i=1}^m u_i x_i)$, and thus define $g_y^m(x_1,\dots,x_m) = g_y(x_1,\dots,x_m, x_{m+1})$. Critical points of $g_y^m$ are characterized by \eqref{e:CPequ_m}. By Proposition~\ref{prop:bc_sol}, as $\alpha\to0$, this equation allows exactly one solution, the barycentric solution $x_y^*$.

The Hessian of $g_y^m$ at $x_y^*$ is now given by $\del_{ij}^2 g_y^m(x_y^*) = -g_y^m(x_y^*)\del_{ij}^2 w_y^m(x_y^*)$ where $w_y^m$ is defined analogously using the parametrization of $S_1$. The dominant term of the evaluated Hessian $-D^2w_y^m(x_y^*)(v,v)$ is
\begin{equation}
 \log(y)\left(
  \Big\langle
   \begin{pmatrix} 
   b_1/(x_y^*)_1^2 & &  \\ 
   & \ddots  & \\ 
    & & b_m/(x_y^*)_m^2 
   \end{pmatrix} v , v
  \Big\rangle 
  + 
  \frac{b_{m+1}}{u_{m+1}^2x_{m+1}^2} 
  \Big\langle 
  \begin{pmatrix} 
   u_1   \\ 
   \vdots  \\  
    u_m 
   \end{pmatrix}
   , v
   \Big\rangle^2
 \right)
 =: \log(y)\Big\langle \eta v,v\Big\rangle 
\end{equation}
whence $-D^2w_y^m(x_y^*)(v,v)\to -\infty$ as $y\to 0$. In particular, $x_y^*$ is the unique and strict global maximum of $g_y^m$ in the small $y$ regime. Notice that the matrix $\eta$ is defined independently of $y$. 

For all $x\neq x_y^*$ we have $w_y^m(x)-w_y^m(x_y^*) = \log(y)(a(x)-a(x_y^*))+\phi_X(x)-\phi_X(x_y^*)\to \infty$ as $y\to0$. 
Outside of a given neighborhood of 
\begin{equation}
 \lim_{y\to0} x_y^* = x^* = \Big(\frac{b_i}{u_i\beta}\Big)_{i=1}^m
\end{equation}
this implies pointwise (and thus uniform by compactness) convergence
\begin{equation}
 \Big|
  g_y^m(x)\exp\Big(w_y^m(x_y^*)\Big) 
  - \exp\Big(
     \frac{1}{2}\log(y)\sum_{i,j=1}^m\eta_{i,j}(x_i - (x_y^*)_i)(x_j - (x_y^*)_j)
  \Big)
 \Big| \to 0
\end{equation}
for $x = (x_1,\dots,x_m)$, as $y\to 0$. Therefore, as $y\to0$ 
\begin{align}
 \EV_{\mu_y}\Big[\xi\Big]
 &=
 Z_y^{-1}\int_{\mathcal{S}_1}\xi(x)g_y^m(x)\,dx \\ 
 \notag
 &\sim
 N_y^{-1}\int_{\mathcal{S}_1}\xi(x) \exp\Big(
     \frac{1}{2}\log(y)\sum_{i,j=1}^m\eta_{i,j}(x_i - (x_y^*)_i)(x_j - (x_y^*)_j)
  \Big)\,dx 
\end{align}
for test functions $\xi$ and where $N_y^{-1}\sim \frac{(-\log(y))^{m/2}\det(\eta)^{1/2}}{(2\pi)^{m/2}}$ denotes the normalization of the normal distribution $N(x_y^*, -\log(y)^{-1}\eta^{-1})$ restricted to   $\mathcal{S}_1$. 
In particular, for the $i$-th coordinate $X_i = \vv<e_i, X>$, $i=1,\dots,m$, this implies
\begin{equation}
\label{e:E[X_i]}
 \lim_{y\to0}\EV_{\mu_y}[X_i] = \frac{b_i}{u_i \beta} .
\end{equation}

\subsection{Universality at the center solution}
Now we constrain $u$, which has to satisfy $u_i>0$, to the hyperplane  $\mathcal{U} := \{u_i>0, \sum_{i=1}^{m+1} u_i = 1\}$. Otherwise there would be ambiguous mixing between the scales of $X$ and $u$.

Theorem~\ref{thm:conc} implies that for every compact subset of $\mathcal{U}$ there is an $\alpha_0>0$ such that $q_{\alpha}$ is strictly concave for all $\alpha\le\alpha_0$ on this subset. For the rest of this section we assume that $\alpha\le\alpha_0$ holds for a fixed (large) subset.  Thus there exists a unique global and strict maximizer $u_{\alpha}^*$ of $q_{\alpha}$ for each $\alpha$. 

For given (small) $\eps>0$, we parametrize $u = (u_1, \dots, u_m, u_{m+1}) = (t_1, \dots, t_m, 1-\sum_i t_i) =: \gamma(t)$ with $t\in\mathcal{U}_m^{\eps} := \{t\in\R^m: t_i\ge\eps, \sum_i t_i\le 1-\eps\}$. We are interested in the shape of
\begin{equation}
 \Sigma_{\alpha}^{\eps} 
 := \Big\{
  \Big(t, q_{\alpha}(\gamma(t))/q_{\alpha}(u_{\alpha}^*)\Big): t\in\mathcal{U}_m^{\eps}
 \Big\}
 \subset \mathcal{L}^{\eps}
\end{equation}
where $\mathcal{L}^{\eps} := \mathcal{U}_m^{\eps}\times\R_{>0}$. We view $\Sigma_{\alpha}^{\eps}$ as a submanifold of $\mathcal{L}^{\eps}$, more precisely it is a section of the line bundle $\pi: \mathcal{L}^{\eps}\to\mathcal{U}_m^{\eps}$, $\pi(t,q) = t$. For given $t\in\mathcal{U}_m^{\eps}$ the parametrization $\gamma$ can be used to determine the tangent space to $\Sigma_{\alpha}^{\eps}$ at $q_{\alpha}(\gamma(t))$, namely we use the kinematic representation of tangent vectors. That is, 
\begin{align}
\label{e:kin_VF}
 \frac{\delta(q_{\alpha}\circ\gamma)}{\delta t_i}
 &= 
 \Big\langle  
  \frac{\delta q_{\alpha}}{\delta u}(q_{\alpha}(\gamma(t))), \frac{\delta\gamma}{\delta t_i}
 \Big\rangle  
 = 
 \EV_{q_{\alpha}(\gamma(t))}\Big[ X_i \Big]
 -  \EV_{q_{\alpha}(\gamma(t))}\Big[ X_{m+1} \Big] 
 \\ 
 \notag
 &
 =
 q_{\alpha}(\gamma(t))\,\EV_{\mu_{q_{\alpha}(\gamma(t))}}\Big[ X_i - X_{m+1} \Big] \\
 \notag 
 &= 
 q_{\alpha}(\gamma(t)) Z_{q_{\alpha}(\gamma(t))}^{-1}
 \int_{\mathcal{S}_1}\Big( x_i - u_{m+1}^{-1}(1 - \sum_{j=1}^m u_j x_j)\Big) g_{q_{\alpha}(\gamma(t))}^m(x)\,dx \\
 \notag 
 &\sim 
 q_{\alpha}(\gamma(t)) 
 N_{q_{\alpha}(\gamma(t))}^{-1}\int_{\mathcal{S}_1}\Big( x_i - \gamma_{m+1}(t)^{-1}(1 - \sum_{j=1}^m t_j x_j)\Big)  \\
 \notag 
 &\phantom{XXXXX}
 \exp\Big(
     \frac{1}{2}\log(y)\sum_{i,j=1}^m\eta_{i,j}(x_i - (x_y^*)_i)(x_j - (x_y^*)_j)
  \Big)\,dx  
 \\ 
 \notag 
 &\sim
 q_{\alpha}(\gamma(t))\Big(
  \frac{b_i}{t_i\beta} - \frac{b_{m+1}}{\gamma_{m+1}(t) \beta}
  + R_{\alpha}(t)
  \Big)
\end{align}
where $\lim_{\alpha\to0}R_{\alpha}(t) = 0$ uniformly in $t$ by \eqref{e:E[X_i]}. Taking $\alpha\to0$ thus defines a family, $\mathcal{D}^{\eps} = \{\xi_1,\dots,\xi_m\}$,  of vector fields $\xi_i: \mathcal{L}^{\eps}\to\R^{m+1}$, 
\begin{equation}
\label{e:univ_VF}
 \xi_i(t,q) =
 \Big( e_i, \xi_i^v(t, q) \Big),
 \qquad 
 \xi_i^v(t, q) 
 = 
 q
 \Big(
  \frac{b_i}{t_i\beta} - \frac{b_{m+1}}{\gamma_{m+1}(t) \beta}
  \Big)
\end{equation}
where $e_i$ is the $i$-th standard basis vector. 
Let 
\begin{equation}
\label{e:Sigma_0}
 \Sigma_0^{\eps}
 = \Big\{
    \Big(t, \lim_{\alpha\to0}\frac{q_{\alpha}(\gamma(t))}{q_{\alpha}(u_{\alpha}^*)} \Big): t\in\mathcal{U}_m^{\eps}
  \Big\}
\end{equation}
which is well-defined as a subset of $\mathcal{L}_m^{\eps}$. The maximizer $u_{\alpha}^* = \gamma(t_{\alpha}^*)$ is characterized as the point where the kinematic tangent vector \eqref{e:kin_VF} vanishes. As $\alpha\to0$ this is thus the (barycentric) point 
\begin{equation}
\label{e:univ_bc}
 \lim_{\alpha\to0} t_{\alpha}^*
 = t^* = \Big(\frac{b_i}{\beta}\Big)_{1=1}^m
\end{equation}
and we note that $u^*_{m+1} = 1-\sum_{i=1}^m b_i/\beta = b_{m+1}/\beta$. 

The system $\mathcal{D}^{\eps}$ consists of pairwise commuting vector fields, i.e.\ the Lie bracket of two such fields satisfies $[\xi_i, \xi_j] = 0$. Thus the system is integrable in the sense of Frobenius. In our case, the leaf $\tilde{\Sigma}_0^{\eps} = \{(t, q_0(t) ): t\in\mathcal{U}_m^{\eps}\} $ passing through the maximizer with normalized maximum, i.e.\ through $(b_1/\beta, \dots, b_m/\beta, 1)$, can be explicitly parametrized as
\begin{equation}
\label{e:univ}
 q_0(t_1,\dots, t_m)
 = 
  \exp\Big(
    \frac{1}{\beta}\Big(
      \sum_{i=1}^m b_i\log\Big(\frac{\beta t_i}{b_i}\Big)
      + b_{m+1}\log\Big(\frac{\beta\gamma_{m+1}(t)}{b_{m+1}}\Big)
    \Big)
  \Big)
\end{equation}
By smooth dependence on intitial conditions it follows that $\tilde{\Sigma}_0^{\eps} = \Sigma_0^{\eps}$ which provides the latter with the structure of a smooth submanifold. 
See Figure~\ref{fig:univ_1D} for the case $m=1$. 

Let $k_{\alpha} = q_{\alpha}(u_{\alpha}^*)$ where  $u_{\alpha}^*$ is the strict global maximizer at level $\alpha$. 
As $\alpha\to0$ definition~\eqref{e:Sigma_0} now implies that 
\begin{equation}
 \label{e:univ_prod}
 q_{\alpha}(u) \sim  k_{\alpha} q_0(u)
\end{equation}
where $u=\gamma(t)$. 

\begin{remark}
Equation~\eqref{e:univ_prod} separates scale and shape. 
The shape function $q_0$ is universal in the sense that it depends only on the covariance structure of the Gaussian log-factors, but not on the position in the (deep) tail and it forgets the mean value, $\mu$, of the log-factors. However, the speed of convergence towards the scale separating regime does depend on $\mu$.  
\end{remark}

\begin{figure}[htbp] 
\centering 
\includegraphics[width=0.45\textwidth]{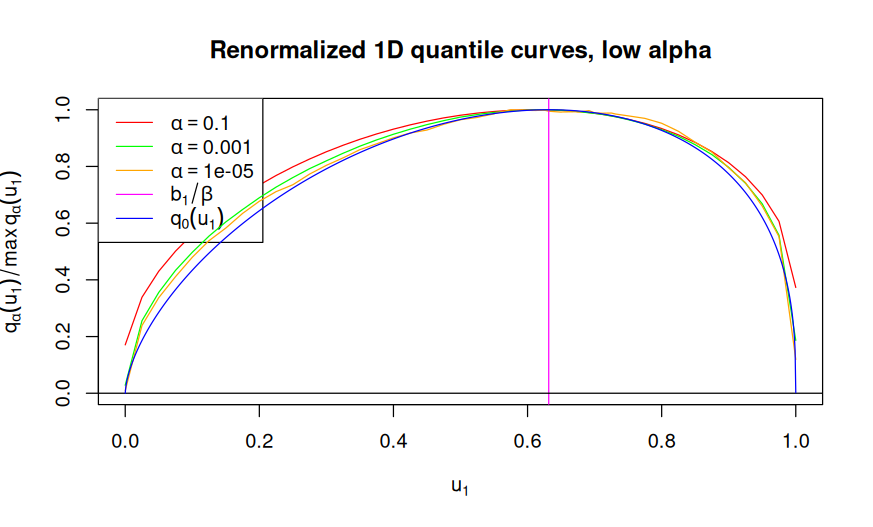}
\qquad
\includegraphics[width=0.45\textwidth]{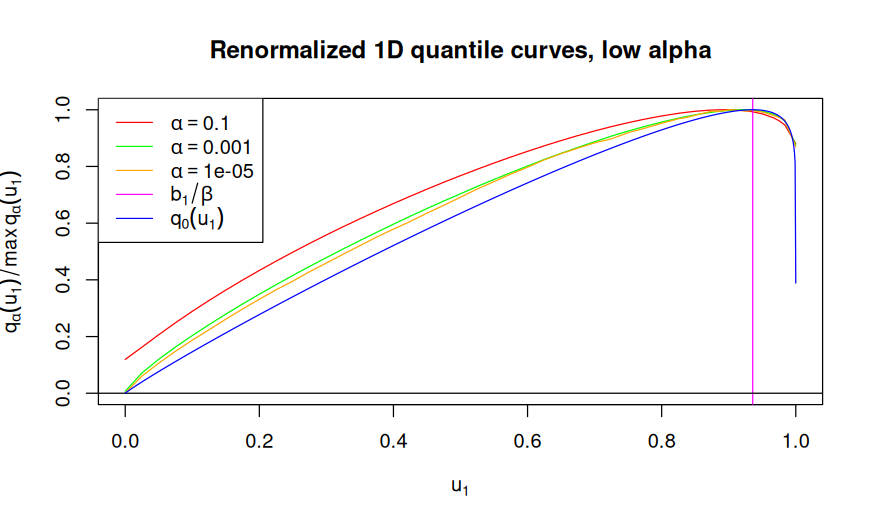}
\vspace{-3mm}
\caption{
The case $m=1$. 
$\mu = (0, 0)$, $B_{11} = 0.5$ (left panel),  $B_{11} = 5$ (right panel), $B_{12} = 0.1$, $B_{22} = 0.25$. The vertical line is drawn at $b_1/\beta$. 
The blue curve is the graph of \eqref{e:univ}. 
Quantiles were generated with $10^7$ trials. 
}
\label{fig:univ_1D} 
\end{figure}


\subsection{Universality at the corner solution ($b_2<0$, $\alpha\to0$)} 
Let $m=2$ and $b_2<0$. With the same notation as above and using the asymptotic $\delta$-concentration at \eqref{e:asym_RC} of $\mu_{q_{\alpha}(t_1)}$ as $\alpha\to0$, we find that 
\begin{equation}
 \lim_{\alpha\to0} \Big(
  q_{\alpha}(\gamma(t_1))
 \Big)^{-1}\frac{\delta(q_{\alpha}\circ\gamma)}{\delta t_1}
 = \frac{t_1^{-1} - 1}{\gamma_2(t_1)}
 = \frac{1}{t_1}
\end{equation}
The asymptotically universal curve $q_0 = q_0(t_1)$ thus satisfies $q_0'(t_1) = q_0(t_1)/t_1$, whence $q_0(t_1) = k t_1$ with $k\ge0$. Again we obtain a one-parameter family of (now linear) shape curves
\begin{equation}\label{e:univ_lin_0}
q_{\alpha}(t_1) \sim k_{\alpha} \cdot t_1
\end{equation}
with $k_{\alpha} = q_{\alpha}(1,0)$.

\subsection{Universality at the corner solution ($b_1>b_2$, $\alpha\to 1$)} 
Let $m=2$ and $b_1 > b_2$. If $b_2>0$ we assume that the global maximum condition \eqref{e:GMC} holds. With the same notation as above and using the asymptotic $\delta$-concentration at \eqref{e:asym_log0} of $\mu_{q_{\alpha}(t_1)}$ as $\alpha\to1$, we find that 
\begin{equation}
 \lim_{\alpha\to1} \Big(
  q_{\alpha}(\gamma(t_1))
 \Big)^{-1}\frac{\delta(q_{\alpha}\circ\gamma)}{\delta t_1}
 = \frac{ - 1}{\gamma_2(t_1)}
 = -\frac{1}{1-t_1}
\end{equation}
The asymptotically universal curve $q_0 = q_0(t_1)$ thus satisfies $q_0'(t_1) = -q_0(t_1)/(1-t_1)$, whence $q_0(t_1) = k (1-t_1)$ with $k\ge0$. Again we obtain a one-parameter family of (now linear) shape curves
\begin{equation}
\label{e:univ_lin_1}
 q_{\alpha}(t_1) \sim k_{\alpha}\cdot (1 - t_1)
\end{equation}
with $k_{\alpha} = q_{\alpha}(0,1)$.
See Figure~\ref{fig:univ_corner} for an illustration of the two corner cases.

\begin{figure}[htbp] 
\centering 
\includegraphics[width=0.45\textwidth]{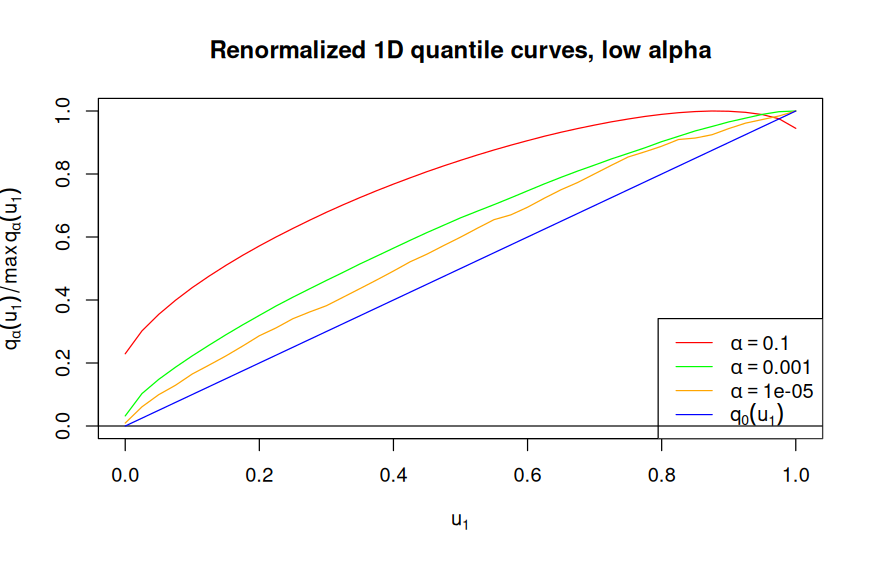}
\qquad
\includegraphics[width=0.45\textwidth]{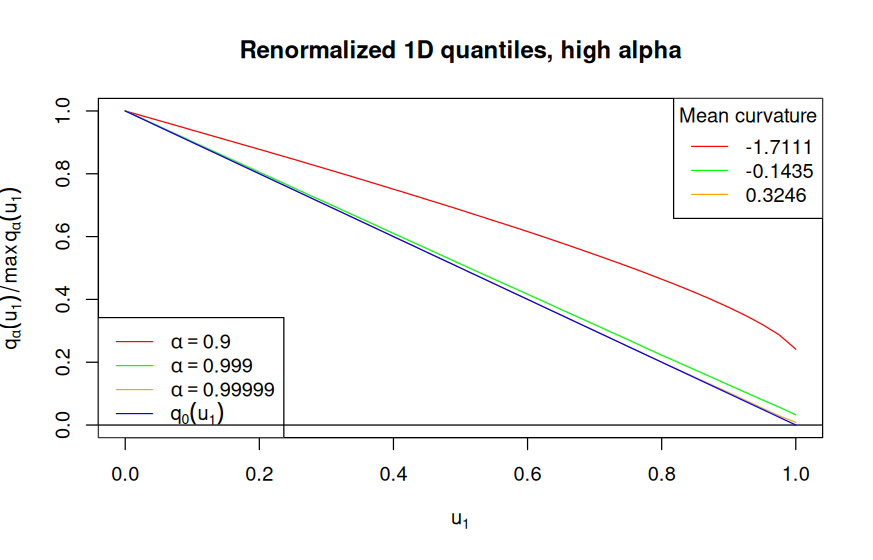}
\vspace{-3mm}
\caption{
The case $m=1$ at the corners for the left and the right tail. 
$\mu = (0, 0)$, $B_{11} = 0.5$, $B_{12} = -0.3$, $B_{22} = 0.25$. The blue lines are the line segments \eqref{e:univ_lin_0} and \eqref{e:univ_lin_1}. In this example the right tail develops convexity only at the last simulated instance (orange curve), as indicated by the mean curvature observations. 
Quantiles were generated with $10^7$ trials. 
}
\label{fig:univ_corner} 
\end{figure}


\section{Optimization of Value-at-Risk}\label{sec:OptVaR}

For a linear portfolio, with pay-off $h(u,X) = u_1 X_1 + u_2 X_2$ and lognormal risk factors, the deep tail quantile optimizer can now be read off from the universal shape in Section~\ref{sec:univ}. 
The case $\alpha\to0$, representing the loss tail of $h(u,X)$, is the case which is interesting for the practitioner. 

Let the density of $(X_1, X_2)$ be given by \eqref{e:log_dens}. Since the deep tail shape curves are independent of the location parameter, we do not need to specify $\mu$. Let $A$ denote the dispersion matrix and assume $\det A = 1$.  

The control $u$ represents investment and  we apply the constraints $u_1+u_2 = 1$ and $u_1\ge0$, $u_2\ge0$.  

For the loss tail, $\alpha\to0$, there are two possible cases. The first case is given by $b_1 = A_{22} - A_{12} > 0$ and $b_2 = A_{11} - A_{12} > 0$. This is the case where there exists relevant diversification between the log-factors, that is
\begin{equation}\label{e:diversif}
 \frac{A_{12}}{\sqrt{A_{11}A_{22}}}
 < 
 \frac{\sqrt{A_{ii}}}{\sqrt{A_{11}+A_{22}-A_{ii}}}
\end{equation} 
holds for $i=1$ and $i=2$. 
According to \eqref{e:univ_prod}, the loss minimization problem 
\begin{equation}\label{e:opt2D}
 q_{\alpha}(u) \to MAX
\end{equation}
subject to $u_1+u_2=1$ is solved asymptotically as $\alpha\to0$ by the investment strategy
\begin{equation}
\label{e:weights}
 u_i = \frac{A_{11} + A_{22} - A_{ii} - A_{12}}{A_{11} + A_{22} - 2A_{12}} 
\end{equation}
for $i=1,2$. 
 
\begin{remark}
Formula~\eqref{e:weights} is the minimal variance solution for the portfolio of $\log$-prices, that is it is the solution of $\var[u_1\log X_1 + u_2\log X_2]\to MIN$ with the same constraints on $u_i$. This formula attaches a higher weight to the less risky investment while keeping also a portion of the riskier position for diversification.  
\end{remark}

Let us now consider the case $\alpha\to0$, $b_1 = A_{22} - A_{12} > 0$ and $b_2 = A_{11} - A_{12} < 0$. This means that there is significant positive correlation between the assets, that is the inequality \eqref{e:diversif} is reversed for $i=1$. In this case the risk minimization problem \eqref{e:opt2D} is solved according to \eqref{e:univ_lin_0} by the strategy 
\begin{equation}
\label{e:corr_opt}
 u = (1, 0)
\end{equation}
i.e., total investment in the less risky position.  
 
Concerning the case $\alpha\to 1$, we may also ask which strategy would maximize profit at a given quantile level. 
Assume the ordering is such that $b_1>b_2$ and if $b_2>0$ we assume additionally that the global maximum condition \eqref{e:GMC} holds. Then \eqref{e:univ_lin_1} implies that the optimizer is found at 
\begin{equation}\label{e:opt_alpha_1}
 u = (0,1)
\end{equation}
i.e., full investment in the more volatile position. 

\begin{remark}
The solutions \eqref{e:weights}, \eqref{e:corr_opt} and \eqref{e:opt_alpha_1} to the respective versions of the optimization problem~\eqref{e:opt2D} correspond to intuition and are not surprising per se. It is, however, remarkable that these strategies are obtained analytically from a quantile perspective.  
\end{remark}

\vspace{7mm}
\noindent\textbf{AI statement.}
Microsoft 365 Copilot was used as an aid in the preparation of the figures. AI was not used to derive results, check proofs or write any parts of this paper.  
\vspace{1mm}

\section*{Appendix: some classical notions in hypersurface geometry}
We follow the presentation in \cite{GallotHulinLafontaine2004}. 
Consider an $m$-dimensional smooth and orientable hypersurface $S\subset \R^{m+1}$. Due to orientability there exists a globally defined normal field $\nu: S\to\R^{m+1}$ with $|\nu_x|=1$ and $\vv<\nu_x, T_xS> = 0$. The \emph{shape operator} $P: TS\to TS$ is defined by 
\begin{equation}
 P(u) = D_u\nu = D\nu\cdot u = \sum u^i\del_i\nu^j e_j
\end{equation}
where $D$ denotes, both, the Euclidean covariant derivative and the differential, and we use Euclidean coordinates in the sum. Since $0 = D_u\vv<\nu,\nu>/2 = \vv<P(u),\nu>$, $P$ does take values in $TS$. Moreover, $\vv<P(u),v> - \vv<P(v),u> = \vv<[u,v],\nu> = 0$, where $[.,.]$ is the Lie bracket of vector fields, whence $P$ is symmetric. 

We mention that the shape operator is related to the second fundamental form $II(u,v) = D_u v - D^S_u v$ via the Weingarten equation $\vv<P(u),v> = -\vv<II(u,v),\nu>$; here $D^S$ denotes covariant differentiation with respect to the induced Riemannian metric in the submanifold $S$. 

The \emph{mean curvature} of $S$ is defined as the trace of the shape operator, normalized by the dimension, that is 
\begin{equation}
\label{e:mean_curv}
 \kappa 
 = \frac{1}{m}\,\textup{tr}\,(P)
 = \frac{1}{m}\,\textup{div}\,\nu
\end{equation}
where $\textup{div}\,\nu$ is the Euclidean divergence. 

We are interested in the case where $S$ is given as the level set of an equation at a regular value. To this end, let $h: \R^{m+1}\to\R$ be a smooth function and fix $y$ such that $\nabla h(x)\neq 0$ for $x\in S = h^{-1}(y)$. The normal field is now given by $\nu = \nabla h/|\nabla h|$. 

From \cite[Theorem~5.20]{GallotHulinLafontaine2004} it follows that 
\begin{equation}
 \frac{\del}{\del\eps}\Big|_{\eps=0}
 \int_{h^{-1}(y+\eps)}\phi(x)\,d\sigma_{y+\eps}
 = 
 \int_S T[\phi](x)\,d\sigma_y
\end{equation}
where $d\sigma_y$ is the surface element on $S$ and the operator $T$ is defined by
\begin{equation}
\label{e:surf_var}
T\phi = \frac{1}{|\nabla h|} 
\,\textup{div}\Big( 
 \phi \nu 
 \Big) 
 = 
 \frac{1}{|\nabla h|}\Big( 
   \vv<\nu,\nabla\phi> 
   + m\kappa\phi 
 \Big) .
\end{equation}
Here we use that $h: \R^{m+1}\to\R$ is a Riemannian submersion (in a vicinity of $y$) and that the infinitesimal approximation of the horizontal lift induced by the Euclidean metric of $y\mapsto y+\eps$ is given by  $x\mapsto x + \eps \nu/|\nabla h|$.

\printbibliography[title={References}]
 
\end{document}